\pdfoutput=1
\documentclass{amsart}
\usepackage[utf8]{inputenc}

\usepackage[expansion=false]{microtype}

\usepackage{amsmath, amsfonts, amssymb, amsthm, amsopn}
\usepackage{eucal}

\usepackage{vmargin}
\setmarginsrb{3.2cm}{2cm}{3.2cm}{2.5cm}{15pt}{0.5cm}{10pt}{0.7cm}

\usepackage[colorlinks]{hyperref}

\usepackage{tikz}
\usetikzlibrary{matrix,arrows,decorations,decorations.pathreplacing}

\makeatletter
\def\@tipextend{
	\arrowsize=.3pt
	\advance\arrowsize by .5\pgflinewidth
	\pgfarrowsleftextend{-3\arrowsize-.5\pgflinewidth}
	\pgfarrowsrightextend{.5\pgflinewidth}
}
\def\@tipoptions{
	\pgfsetdash{}{0pt}
	\pgfsetroundcap
	\pgfsetroundjoin
}

\pgfarrowsdeclare{tip}{tip}{\@tipextend}{
	\@tipoptions
	\pgfmathsetlength{\pgfutil@tempdima}{\pgflinewidth}
	\pgfpathmoveto {\pgfpoint{-3.5\pgfutil@tempdima}{6.2\pgfutil@tempdima}}
	\pgfpathcurveto{\pgfpoint{-2.8\pgfutil@tempdima}{2\pgfutil@tempdima}}
	               {\pgfpoint{1\pgfutil@tempdima}{0.05\pgfutil@tempdima}}
	               {\pgfpoint{2\pgfutil@tempdima}{0pt}}
	\pgfpathcurveto{\pgfpoint{1\pgfutil@tempdima}{-0.05\pgfutil@tempdima}}
	               {\pgfpoint{-2.8\pgfutil@tempdima}{-2\pgfutil@tempdima}}
	               {\pgfpoint{-3.5\pgfutil@tempdima}{-6.2\pgfutil@tempdima}}
	\pgfusepathqstroke
}
\pgfarrowsdeclare{top}{top}{\@tipextend}{
	\@tipoptions
	\pgfmathsetlength{\pgfutil@tempdima}{\pgflinewidth}
	\pgfpathmoveto {\pgfpoint{-3.5\pgfutil@tempdima}{6.2\pgfutil@tempdima}}
	\pgfpathcurveto{\pgfpoint{-2.8\pgfutil@tempdima}{2\pgfutil@tempdima}}
	               {\pgfpoint{1\pgfutil@tempdima}{0.05\pgfutil@tempdima}}
	               {\pgfpoint{2\pgfutil@tempdima}{0pt}}
	\pgfpathcurveto{\pgfpoint{1\pgfutil@tempdima}{-0.05\pgfutil@tempdima}}
	               {\pgfpoint{-1.95\pgfutil@tempdima}{-1.8\pgfutil@tempdima}}
	               {\pgfpoint{-2.85\pgfutil@tempdima}{-4.2\pgfutil@tempdima}}
	\pgfusepathqstroke
}
\pgfarrowsdeclare{bot}{bot}{\@tipextend}{
	\@tipoptions 
	\pgfmathsetlength{\pgfutil@tempdima}{\pgflinewidth}
	\pgfpathmoveto {\pgfpoint{-2.85\pgfutil@tempdima}{4.2\pgfutil@tempdima}}
	\pgfpathcurveto{\pgfpoint{-1.95\pgfutil@tempdima}{1.8\pgfutil@tempdima}}
	               {\pgfpoint{1\pgfutil@tempdima}{0.05\pgfutil@tempdima}}
	               {\pgfpoint{2\pgfutil@tempdima}{0pt}}
	\pgfpathcurveto{\pgfpoint{1\pgfutil@tempdima}{-0.05\pgfutil@tempdima}}
	               {\pgfpoint{-2.8\pgfutil@tempdima}{-2\pgfutil@tempdima}}
	               {\pgfpoint{-3.5\pgfutil@tempdima}{-6.2\pgfutil@tempdima}}
	\pgfusepathqstroke
}
\pgfarrowsdeclare{mid}{mid}{\@tipextend}{
	\@tipoptions 
	\pgfmathsetlength{\pgfutil@tempdima}{\pgflinewidth}
	\pgfpathmoveto {\pgfpoint{-2.85\pgfutil@tempdima}{4.2\pgfutil@tempdima}}
	\pgfpathcurveto{\pgfpoint{-1.95\pgfutil@tempdima}{1.8\pgfutil@tempdima}}
	               {\pgfpoint{1\pgfutil@tempdima}{0.05\pgfutil@tempdima}}
	               {\pgfpoint{2\pgfutil@tempdima}{0pt}}
	\pgfpathcurveto{\pgfpoint{1\pgfutil@tempdima}{-0.05\pgfutil@tempdima}}
	               {\pgfpoint{-1.95\pgfutil@tempdima}{-1.8\pgfutil@tempdima}}
	               {\pgfpoint{-2.85\pgfutil@tempdima}{-4.2\pgfutil@tempdima}}
	\pgfusepathqstroke
}

\expandafter\let\csname pgf@ar@means@|\endcsname=\relax
\expandafter\let\csname pgf@arrow@code@|\endcsname=\relax
\pgfarrowsdeclare{|}{|}{
	\pgfarrowsleftextend{-0.25\pgflinewidth}
	\pgfarrowsrightextend{.75\pgflinewidth}
}{
	\@tipoptions 
	\pgfmathsetlength{\pgfutil@tempdima}{\pgflinewidth}
	\pgfpathmoveto{\pgfpoint{0pt}{6\pgfutil@tempdima}}
	\pgfpathlineto{\pgfpoint{0pt}{-6\pgfutil@tempdima}}
	\pgfusepathqstroke
}
\makeatother

\pgfdeclaredecoration{normal}{initial}{
	\state{initial}[width=\pgfdecoratedpathlength-1sp]{\pgfpathmoveto{\pgfpointorigin}} 
	\state{final}{\pgfpathlineto{\pgfpointorigin}} 
}

\pgfarrowsdeclarealias{c}{c}{left hook}{right hook}

\def\rowsep{2.5em}
\def\colsep{2.5em}

\def\diagram{\matrix[diagram]}
\def\arrows{\path[->,font=\scriptsize]}

\newenvironment{tikzcentered}{\center\tikzpicture}{\endtikzpicture\endcenter}
\newenvironment{tikzmath}{\displaymath\tikzpicture}{\endtikzpicture\enddisplaymath}
\newenvironment{tikzequation}{\equation\tikzpicture}{\endtikzpicture\endequation}

\tikzset{
	every picture/.style={
		baseline=(current bounding box.center),% equation numbers are centered.
		>=tip,% set default arrow tip
		line cap=round % set round caps
	},
	diagram/.style={
		inner sep=0,% we don't want any additional margin for diagrams (displaymath/center takes care of that)
		execute at begin cell=\node\bgroup\math\displaystyle,% displaystyle math nodes
		execute at end cell={
			\endmath\egroup;% close node
			\coordinate (-\the\pgfmatrixcurrentrow\the\pgfmatrixcurrentcolumn) at ([yshift=2.6] \the\pgfmatrixcurrentrow\the\pgfmatrixcurrentcolumn.base west);% set uniform left horizontal anchor
			\coordinate (\the\pgfmatrixcurrentrow\the\pgfmatrixcurrentcolumn -) at ([yshift=2.6] \the\pgfmatrixcurrentrow\the\pgfmatrixcurrentcolumn.base east);% set uniform right horizontal anchor
		},
		cells={anchor=base},% aligns text across cells
		nodes={
			name=\the\pgfmatrixcurrentrow\the\pgfmatrixcurrentcolumn,% name (ij)
			inner sep=.3333em % default padding (override the matrix's 'inner sep=0')
		},
		row sep=\rowsep,
		column sep=\colsep
	},
	vshift/.style={
		decoration={normal,raise=#1},
		decorate 
	}
}

\newcommand{\resp}{{\sfcode`\.1000 resp.}}
\newcommand{\ie}{{\sfcode`\.1000 i.e.}}

\newcommand{\cf}{{\sfcode`\.1000 cf.}}
\newcommand{\viz}{{\sfcode`\.1000 viz.}}

\numberwithin{equation}{section}

\theoremstyle{plain}
\newtheorem{theorem}[equation]{Theorem}
\newtheorem{proposition}[equation]{Proposition}
\newtheorem{lemma}[equation]{Lemma}
\newtheorem{corollary}[equation]{Corollary}

\theoremstyle{definition}
\newtheorem{definition}[equation]{Definition}

\newtheorem{remark}[equation]{Remark}
\newtheorem{construction}[equation]{Construction}

\let\scr=\mathcal
\let\bb=\mathbf
\let\phi=\varphi

\def\N{\bb N}
\def\Z{\bb Z}
\def\Q{\bb Q}

\def\A{\bb A}
\def\G{\bb G}

\def\1{\bb 1}

\def\Sch{\scr S\mathrm{ch}}

\def\Set{\scr S\mathrm{et}}

\def\et{{\mathrm{\acute et}}}

\def\h{{\mathrm{h}}}
\def\qfh{{\mathrm{qfh}}}

\def\ft{{\mathrm{ft}}}

\def\Nis{{\mathrm{Nis}}}

\def\ind{\mathrm{ind}}

\def\gr{\mathrm{gp}}
\let\gp=\gr
\def\tr{\mathrm{tr}}
\def\op{\mathrm{op}}

\def\suchthat{\:\vert\:}

\let\into=\hookrightarrow

\let\tens=\otimes
\def\FFree{\mathrm{FFree}{}}

\def\HC{\mathrm{HC}}

\def\Mod{\mathrm{Mod}{}}

\def\Cor{\scr C\mathrm{or}}
\def\TopR{\scr T\mathrm{op}_\infty}

\def\Fin{\scr F\mathrm{in}}

\def\Orb{\scr O}
\def\LPi{\overline\Pi{}}

\def\CMon{\mathrm{CMon}}

\def\HC{\mathrm{HC}}

\DeclareMathOperator{\Hom}{Hom}
\DeclareMathOperator{\Cut}{Cut}
\def\Sym{\mathrm{S}}
\DeclareMathOperator{\Map}{Map}
\DeclareMathOperator{\Fun}{Fun}
\DeclareMathOperator{\Pro}{Pro}
\DeclareMathOperator{\Ind}{Ind}

\DeclareMathOperator{\Shv}{Shv}
\DeclareMathOperator{\PSh}{PSh}
\DeclareMathOperator{\car}{char}
\DeclareMathOperator{\Et}{\acute Et{}}
\DeclareMathOperator{\Etd}{\acute Et_\mathit{l}^\times}

\DeclareMathOperator{\Spec}{Spec}

\DeclareMathOperator{\cosk}{cosk}

\DeclareMathOperator{\fpr}{fpr}

\let\lim=\relax
\DeclareMathOperator*{\lim}{lim}
\DeclareMathOperator*{\colim}{colim}
\DeclareMathOperator*{\hocolim}{hocolim}
\DeclareMathOperator*{\holim}{holim}

\title{The étale symmetric Künneth theorem}
\author{Marc Hoyois}
\address{Fakultät für Mathematik\\
Universität Regensburg\\
93040 Regensburg\\
Germany}
\email{\href{mailto:marc.hoyois@ur.de}{marc.hoyois@ur.de}}
\urladdr{\url{http://www.mathematik.ur.de/hoyois/}}

\date{\today}

\begin{document}

\begin{abstract}
	Let $k$ be an algebraically closed field, $l\neq\car k$ a prime number, and $X$ a quasi-projective scheme over $k$. We show that the étale homotopy type of the $d$th symmetric power of $X$ is $\Z/l$-homologically equivalent to the $d$th strict symmetric power of the étale homotopy type of $X$. We deduce that the $\Z/l$-local étale homotopy type of a motivic Eilenberg--Mac Lane space is an ordinary Eilenberg--Mac Lane space.
\end{abstract}

\setcounter{tocdepth}{1}

\maketitle

\tableofcontents

\section*{Introduction} 
\label{sec:introduction}

In the first part of this paper we show that the étale homotopy type of the $d$th symmetric power of a quasi-projective scheme $X$ over a separably closed field $k$ is $\Z/l$-homologically equivalent to the $d$th symmetric power of the étale homotopy type of $X$, where $l\neq\car k$ is any prime. Symbolically,
\begin{equation*}\tag{\textasteriskcentered}
	\label{eqn:intro}
	L_{\Z/l}\Pi_\infty^\et(\Sym^dX)\simeq L_{\Z/l}\Sym^d\Pi_\infty^\et (X),
\end{equation*}
where $\Pi_\infty^\et$ is the étale homotopy type, $\Sym^d$ is the $d$th symmetric power (more precisely the \emph{strict} symmetric power), and $L_{\Z/l}$ is $\Z/l$-localization à la Bousfield--Kan.
The étale homotopy type $\Pi_\infty^\et X$ of a scheme $X$ is a pro-space originally defined by Artin and Mazur \cite{AM} and later refined by Friedlander \cite{Friedlander:1982}. It is characterized by the property that the (nonabelian) étale cohomology of $X$ with constant coefficients coincides with the cohomology of $\Pi_\infty^\et X$. 

The formula \eqref{eqn:intro} is related to a theorem of Deligne about the étale cohomology of symmetric powers \cite[XVII, Théorème 5.5.21]{SGA4}, but there are three significant differences:
\begin{enumerate}
	\item Deligne's theorem is about cohomology \emph{with proper support}, and so does not say anything about the cohomology of non-proper schemes.
	\item We give an equivalence at the level of homotopy types, whereas Deligne only gives an equivalence at the level of cochains.
	\item Deligne's theorem works over an arbitrary quasi-compact quasi-separated base and with arbitrary Noetherian torsion coefficients; for our result the base must be a separably closed field whose characteristic is prime to the torsion of the coefficients.
\end{enumerate}
 While there may be a relative version of~\eqref{eqn:intro} over a base, the localization away from the residual characteristics cannot be avoided when dealing with non-proper schemes.
 
 In his proof, after reducing to the case where the base is a field $k$, Deligne employs Witt vectors to further reduce to the case where $k$ has characteristic zero (concluding with a transcendental argument). In this step it is crucial that $X$ be proper over $k$. 
 Our arguments are thus necessarily quite different. We use the existence a schematic topology, finer than the étale topology but cohomologically equivalent to it, for which the quotient map $X^d\to \Sym^d X$ is a covering; this is the qfh topology used extensively by Voevodsky in his work on triangulated categories of motives.

Combining~\eqref{eqn:intro} with the motivic Dold--Thom theorem, we show that if $k$ is algebraically closed and $A$ is an abelian group on which the characteristic exponent of $k$ acts invertibly, then the $\Z/l$-local étale homotopy type of a motivic Eilenberg--Mac Lane space $K(A(q),p)$ is the $\Z/l$-localization of an ordinary Eilenberg--Mac Lane space $K(A,p)$.

\subsection*{Conventions}
Throughout this paper, we use the language of $\infty$-categories developed in \cite{HTT} and \cite{HA}.
Although our main results can be stated in more classical language, their proofs use the flexibility of higher category theory in an essential way. We warn the reader that this is the \emph{default} language in this paper, so for example the word ``colimit'' always means ``homotopy colimit'', ``unique'' means ``unique up to a contractible space of choices'', etc. 
We will use the following notation:
\begin{itemize}
	\item $\scr S$ is the $\infty$-category of small $\infty$-groupoids, which we also call \emph{spaces};
	\item $\Set_\Delta$ is the category of simplicial sets;
	\item $\TopR$ is the $\infty$-category of $\infty$-topoi and geometric morphisms \cite[\S6.3]{HTT};
	\item $\scr C_{\leq n}$ is the subcategory of $n$-truncated objects in an $\infty$-category $\scr C$;
	\item $\scr C^\omega$ is the subcategory of compact objects in an $\infty$-category $\scr C$ with filtered colimits;
	\item $\h\scr C$ is the homotopy $1$-category of an $\infty$-category $\scr C$;
	\item $\scr X^\wedge$ is the hypercompletion of an $\infty$-topos $\scr X$ \cite[\S6.5.2]{HTT}.
\end{itemize}

\subsection*{Historical note}
The first draft of this paper was written in 2011 as a step towards the computation of the motivic Steenrod algebra in positive characteristic. Afterwards I realized that the technology of étale homotopy types could be avoided completely using the Bloch–Kato conjecture, which was the approach taken in \cite{HKO}. Since I had no other application in mind I did not attempt to turn this draft into a publishable paper. More recently however, the main result of this paper was used by Zargar in \cite{Zargar} to compute the weight $0$ homotopy groups of the motivic sphere spectrum in positive characteristic. Given this new application, it seemed important that this paper be published after all. I want to thank Chuck Weibel for encouraging me to finally take this paper out of its draft state.

\section{Homotopy types of schemes} 
\label{sec:homotopy_types_of_schemes}

Let $\tau$ be a pretopology on the category of schemes (in the sense of \cite[II, Définition 1.3]{SGA4}). 
 If $X$ is a scheme, the small $\tau$-site of $X$ is the full subcategory of $\Sch_X$ spanned by the members of the $\tau$-coverings of $X$ and equipped with the Grothendieck topology induced by $\tau$ (we assume that this is an essentially small category). We denote by $X_\tau$ the $\infty$-topos of sheaves of spaces on the small $\tau$-site of $X$.
 The assignment $X\mapsto X_\tau$ is functorial: a morphism of schemes $f\colon X\to Y$ induces a geometric morphism of $\infty$-topoi $f_\ast\colon X_\tau\to Y_\tau$ given by $f_*(\scr F)(U)=\scr F(U\times_YX)$. 

Recall that the functor $\scr S\to\TopR$ associating to an $\infty$-groupoid its classifying $\infty$-topos admits a pro-left adjoint $\Pi_\infty\colon\TopR\to \Pro(\scr S)$ associating to any $\infty$-topos its \emph{shape} (see \cite[\S7.1.6]{HTT} or \cite{HoyoisGalois}). The \emph{$\tau$-homotopy type} $\Pi_\infty^\tau X$ of a scheme $X$ is the shape of the $\infty$-topos $X_\tau$:
\[\Pi_\infty^\tau X=\Pi_\infty (X_\tau).\]
This construction defines a functor $\Pi_\infty^\tau\colon\Sch\to\Pro(\scr S)$.

Let $\scr X$ be an $\infty$-topos and let $c\colon\scr S\to\scr X$ be the constant sheaf functor. By definition of the shape, we have $\Map_{\scr X}(*,cK)\simeq \Map_{\Pro(\scr S)}(\Pi_\infty\scr X, K)$ for every $K\in\scr S$. In particular, the cohomology of $\scr X$ with coefficients in an abelian group can be computed as the continuous cohomology of the pro-space $\Pi_\infty\scr X$. If $\scr X$ is locally connected (\ie, if for every $X\in\scr X$ the pro-set $\pi_0\Pi_\infty (\scr X_{/X})$ is constant), we have more generally that the category $\Fun(\Pi_\infty\scr X,\Set)$ of discrete local systems on $\Pi_\infty\scr X$ is equivalent to the category of locally constant sheaves of sets on $\scr X$ \cite[Theorem 3.13]{HoyoisGalois}, and, if $\scr A$ is such a sheaf of abelian groups, then $H^\ast(\scr X,\scr A)$ coincides with the continuous cohomology of the pro-space $\Pi_\infty\scr X$ with coefficients in the corresponding local system \cite[Proposition 2.15]{HoyoisGalois}.

\begin{remark}\label{rmk:site}
	In the definition of $\Pi_\infty^\tau X$, we could have used any $\tau$-site of $X$-schemes containing the small one. For if $X_{\tau}'$ is the resulting $\infty$-topos of sheaves, the canonical geometric morphism $X_\tau\to X_{\tau}'$ is obviously a shape equivalence. It follows that the functor $\Pi_\infty^\tau$ depends only on the Grothendieck topology induced by $\tau$.
\end{remark}

\begin{remark}\label{rmk:relative}
	For schemes over a fixed base scheme $S$, we can define in the same way a relative version of the $\tau$-homotopy type functor taking values in the $\infty$-category $\Pro(S_\tau)$. 
\end{remark}

\begin{remark}\label{rmk:qwe}
	Let $X_\tau^\wedge$ be the hypercompletion of $X_\tau$. By the generalized Verdier hypercovering theorem \cite[Theorem 7.6(b)]{DHI}, $\Pi_\infty (X_\tau^\wedge)$ is corepresented by the simplicially enriched diagram $\Pi_0\colon \HC_\tau(X)\to\Set_\Delta$ where $\HC_\tau(X)$ is the cofiltered simplicial category of $\tau$-hypercovers of $X$ and $\Pi_0(U_\bullet)$ is the simplicial set that has in degree $n$ the colimit of the presheaf $U_n$. See \cite[\S 5]{HoyoisGalois} for details.
\end{remark}

\begin{remark}
	When $\tau=\et$ is the étale pretopology and $X$ is locally Noetherian, $\Pi_\infty (X_\et^\wedge)$ is corepresented by the étale topological type defined by Friedlander in \cite[\S4]{Friedlander:1982}. This follows easily from Remark~\ref{rmk:qwe}.
\end{remark}

\begin{lemma}\label{lem:descent}
	Let $U$ be a diagram in the small $\tau$-site of a scheme $X$ whose colimit in $X_\tau$ is $X$. Then $\Pi_\infty^\tau X$ is the colimit of the diagram of pro-spaces $\Pi_\infty^\tau U$.
\end{lemma}

\begin{proof}
	By \cite[Proposition 6.3.5.14]{HTT}, the $\infty$-topos $X_\tau$ is the colimit in $\TopR$ of the diagram of $\infty$-topoi $U_\tau$. Since $\Pi_\infty\colon\TopR\to\Pro(\scr S)$ is left adjoint, it preserves this colimit.
\end{proof}

\begin{remark}
	Similarly, if $U_\bullet\to X$ is a representable $\tau$-hypercover of $X$, then it is a colimit diagram in $X_\tau^\wedge$, so that $\Pi_\infty (X_\tau^\wedge)$ is the colimit of the simplicial pro-space $\Pi_\infty((U_\bullet)_\tau^\wedge)$. The trivial proof of this fact can be compared with the rather technical proof of~\cite[Theorem 3.4]{Isaksen:2004}, which is the special case $\tau=\et$. This is a good example of the usefulness of the topos-theoretic definition of the shape.
\end{remark}

\section{Strict symmetric powers in \texorpdfstring{$\infty$}{infinity}-categories} 
\label{sec:symmetric_powers}

If $X$ is a CW complex, its $d$th symmetric power $\Sym^dX$ is the set of orbits of the action of the symmetric group $\Sigma_d$ on $X^d$, endowed with the quotient topology. Even though the action of $\Sigma_d$ on $X^d$ is not free, 
it is well known that the homotopy type of $\Sym^d X$ is an invariant of the homotopy type of $X$.
More generally, if $G$ is a group acting on a CW complex $X$, the orbit space $X/G$ can be written as the homotopy colimit
\begin{equation}\label{eqn:X/G}
	X/G\simeq \hocolim_{H\in\Orb(G)^\op}X^H,
\end{equation}
where $\Orb(G)$ is the orbit category of $G$ (whose objects are the subgroups of $G$ and whose morphisms are the $G$-equivariant maps between the corresponding quotients) and $X^H$ is the subspace of $H$-fixed points \cite[Chapter 4, Lemma A.3]{Farjoun:1996}.
In the case of the symmetric group $\Sigma_d$ acting on $X^d$, if $H$ is a subgroup of $\Sigma_d$, then $(X^d)^H\simeq X^{o(H)}$ where $o(H)$ is the set of orbits of the action of $H$ on $\{1,\dotsc,d\}$ and where the factor of $X^{o(H)}$ indexed by an orbit $\{i_1,\dotsc, i_r\}$ is sent diagonally into the corresponding $r$ factors of $X^d$. The formula~\eqref{eqn:X/G} becomes
\begin{equation*}\label{eqn:sympow}
	\Sym^d X\simeq \hocolim_{H\in\Orb(\Sigma_d)^\op}X^{o(H)}.
\end{equation*}
This shows that $\Sym^d$ preserves homotopy equivalences between CW complexes. In particular, it induces a functor $\Sym^d$ from the $\infty$-category of spaces to itself.

This motivates the following definition:

\begin{definition}
	Let $\scr C$ be an $\infty$-category with colimits and finite products and $d\geq 0$ an integer. The $d$th \emph{strict symmetric power} of $X\in\scr C$ is
	\[
	\Sym^d X = \colim_{H\in\Orb(\Sigma_d)^\op}X^{o(H)}.
	\]
\end{definition}

We will relate strict symmetric powers to the notion of strictly commutative monoid in \S\ref{sec:group_completions}.
Note that $\Sym^0X$ is a final object of $\scr C$ and that $\Sym^1 X\simeq X$. 
For example, in an $\infty$-category of sheaves of spaces on a site, $\Sym^d$ is computed by applying $\Sym^d$ objectwise and sheafifying the result, and in a $1$-category it is the usual symmetric power, namely the coequalizer of the action groupoid $\Sigma_d\times X^d\rightrightarrows X^d$. We note that any functor that preserves colimits and finite products commutes with $\Sym^d$.

\begin{remark}
	If the product in $\scr C$ preserves sifted colimits in each variable (for example, if $\scr C$ is cartesian closed or projectively generated), it follows from \cite[Lemma 5.5.8.11]{HTT} that the functor $\Sym^d\colon \scr C\to \scr C$ preserves sifted colimits. In particular, $\Sym^d\colon \scr S\to \scr S$ is the left Kan extension of the functor $\Sym^d\colon \Fin\to \Fin$, where $\Fin$ is the category of finite sets.
\end{remark}

\begin{remark}
	 More generally, one has a strict symmetric power $\Sym^\phi X$ for any group homomorphism $\phi\colon G\to\Sigma_d$:
	 \[
	 \Sym^\phi X = \colim_{H\in\Orb(G)^\op}X^{o(H)}.
	 \]
\end{remark}

\begin{lemma}\label{lem:symmtopos}
	Let $\scr X$ be an $\infty$-topos. Then the inclusion $\scr X_{\leq 0}\into \scr X$ of the subcategory of discrete objects preserves strict symmetric powers.
\end{lemma}

\begin{proof}
	This is true if $\scr X=\scr S$ since symmetric powers preserve discrete CW complexes, hence if $\scr X$ is a presheaf $\infty$-topos. It remains to observe that if $a\colon\scr X\to\scr Y$ is a left exact localization and the result is true in $\scr X$, then it is true in $\scr Y$: this follows from the fact that $a$ preserves $0$-truncated objects \cite[Proposition 5.5.6.16]{HTT}.
\end{proof}

\section{Homological localizations of pro-spaces} 
\label{sec:homological_localizations}

Let $\Pro(\scr S)$ denote the $\infty$-category of pro-spaces. Recall that this is the $\infty$-category freely generated by $\scr S$ under cofiltered limits and that it is equivalent to the full subcategory of $\Fun(\scr S,\scr S)^\op$ spanned by accessible left exact functors \cite[Proposition 3.1.6]{DAG13}. Any such functor is equivalent to $Y\mapsto \colim_{i\in I}\Map(X_i,Y)$ for some small cofiltered diagram $X\colon I\to\scr S$. Moreover, combining~\cite[Proposition 5.3.1.16]{HTT} and the proof of~\cite[Proposition 8.1.6]{SGA4}, we can always find such a corepresentation where $I$ is a cofiltered poset such that, for each $i\in I$, there are only finitely many $j$ with $i\leq j$; such a poset is called \emph{cofinite}.

In~\cite{Isaksen:2004b}, Isaksen constructs a proper model structure on the category $\Pro(\Set_\Delta)$ of pro-simplicial sets, called the \emph{strict model structure}, with the following properties:
\begin{itemize}
	\item a pro-simplicial set $X$ is fibrant if and only if it is isomorphic to a diagram $(X_s)_{s\in I}$ such that $I$ is a cofinite cofiltered poset and
	$X_s\to \lim_{s<t}X_t$
	is a Kan fibration for all $s\in I$;
	\item the inclusion $\Set_\Delta\into\Pro(\Set_\Delta)$ is a left Quillen functor;
	\item it is a simplicial model structure with simplicial mapping sets defined by
	\[\Map_\Delta(X,Y)=\Hom(X\times\Delta^\bullet,Y).\]
\end{itemize}
Denote by $\Pro'(\scr S)$ the $\infty$-category associated to this model category, and by $c\colon\scr S\to\Pro'(\scr S)$ the left derived functor of the inclusion $\Set_\Delta\into\Pro(\Set_\Delta)$. Since $\Pro'(\scr S)$ admits cofiltered limits, there is a unique functor $\phi\colon\Pro(\scr S)\to\Pro'(\scr S)$ that preserves cofiltered limits and such that $\phi\circ j\simeq c$, where $j\colon\scr S\into\Pro(\scr S)$ is the Yoneda embedding.

\begin{lemma}\label{lem:isaksen}
	$\phi\colon\Pro(\scr S)\to\Pro'(\scr S)$ is an equivalence of $\infty$-categories.
\end{lemma}

\begin{proof}
	Let $X\in\Pro(\Set_\Delta)$ be fibrant. Then $X$ is isomorphic to a diagram $(X_s)$ indexed by a cofinite cofiltered poset and such that $X_s\to\lim_{s<t}X_t$ is a Kan fibration for all $s$, and so, for all $Z\in\Pro(\Set_\Delta)$, $\Map_\Delta(Z,X_s)\to\lim_{s<t}\Map_\Delta(Z,X_t)$ is a Kan fibration. It follows that the limit $\Map_\Delta(Z,X)\simeq\lim_s\Map_\Delta(Z,X_s)$ in $\Set_\Delta$ is in fact a limit in $\scr S$, so that $X\simeq \lim_s c(X_s)$ in $\Pro'(\scr S)$. This shows that $\phi$ is essentially surjective.
	
	Let $X\in\Pro(\scr S)$ and choose a corepresentation $X\colon I\to\scr S$ where $I$ is a cofinite cofiltered poset. Using the model structure on $\Set_\Delta$, $X$ can be strictified to a diagram $X'\colon I\to\Set_\Delta$ such that $X'_s\to\lim_{s<t} X'_t$ is a Kan fibration for all $s\in I$. By the first part of the proof, we then have $X'\simeq \lim_s c(X'_s)$ in $\Pro'(\scr S)$, whence $\phi X\simeq X'$. Given also $Y\in\Pro(\scr S)$, we have
	\[\Map(X',Y')\simeq
\lim_t\Map(X',cY'_t)\simeq\holim_t\Map_\Delta(X',Y'_t)\simeq\lim_t\colim_s\Map(X'_s, Y'_t),\]
	where in the last step we used that filtered colimits of simplicial sets are always colimits in $\scr S$. This shows that $\phi$ is fully faithful.
\end{proof}

Let $\scr S_{<\infty}\subset \scr S$ be the $\infty$-category of truncated spaces. A \emph{pro-truncated space} is a pro-object in $\scr S_{<\infty}$. It is clear that the full embedding $\Pro(\scr S_{<\infty})\into \Pro(\scr S)$ admits a left adjoint
\[\tau_{<\infty}\colon \Pro(\scr S)\to \Pro(\scr S_{<\infty})\]
that preserves cofiltered limits and sends a constant pro-space to its Postnikov tower; it also preserves finite products since truncations do.
The $\tau_{<\infty}$-equivalences in $\Pro(\scr S_*^{\geq 1})$ are  precisely those maps that become $\natural$-isomorphisms in $\Pro(\h\scr S_*^{\geq 1})$ in the sense of Artin and Mazur \cite[Definition 4.2]{AM}.

\begin{remark}\label{rmk:Isaksen}
	The model structure on $\Pro(\Set_\Delta)$ defined in~\cite{Isaksen:2001} is the left Bousfield localization of the strict model structure at the class of $\tau_{<\infty}$-equivalences. It is therefore a model for the $\infty$-category $\Pro(\scr S_{<\infty})$ of pro-truncated spaces.
\end{remark}

Let $R$ be a commutative ring. A morphism $f\colon X\to Y$ in $\Pro(\scr S)$ is called an \emph{$R$-homological equivalence} (\resp{} an \emph{$R$-cohomological equivalence}) if it induces an equivalence of homology pro-groups $H_\ast(X,R)\simeq H_\ast(Y,R)$ (\resp{} an equivalence of cohomology groups $H^*(Y,R)\simeq H^*(X,R)$). By~\cite[Proposition 5.5]{Isaksen:2005}, $f$ is an $R$-homological equivalence if and only if it induces isomorphisms in cohomology with coefficients in arbitrary $R$-modules. A pro-space $X$ is called \emph{$R$-local} if it is local with respect to the class of $R$-homological equivalences, \ie, if for every $R$-homological equivalence $Y\to Z$ the induced map $\Map(Z,X)\to\Map(Y,X)$ is an equivalence in $\scr S$. A pro-space is called \emph{$R$-profinite}\footnote{This terminology is motivated by the case $R=\Z/l$, see Remark~\ref{rmk:l-profinite}.} if it is local with respect to the class of $R$-cohomological equivalences. 
We denote by $\Pro(\scr S)_R$ (\resp{} $\Pro(\scr S)^R$) the $\infty$-category of $R$-local (\resp{} $R$-profinite) pro-spaces. 

The characterization of $R$-homological equivalences in terms of cohomology shows that any $\tau_{<\infty}$-equivalence is an $R$-homological equivalence. We thus have a chain of full embeddings
\[\Pro(\scr S)^R\subset\Pro(\scr S)_R\subset \Pro(\scr S_{<\infty})\subset \Pro(\scr S).\]

\begin{proposition}\label{prop:finiteprod}
	The inclusions $\Pro(\scr S)_R\subset\Pro(\scr S)$ and $\Pro(\scr S)^R\subset\Pro(\scr S)$ admit left adjoints $L_R$ and $L^R$. Moreover, $L_R$ preserves finite products.
\end{proposition}

\begin{proof}
	The existence of the localization functors $L_R$ and $L^R$ follows from Lemma~\ref{lem:isaksen} and the existence of the corresponding left Bousfield localizations of the strict model structure on $\Pro(\Set_\Delta)$ \cite[Theorems 6.3 and 6.7]{Isaksen:2005}. We also give a self-contained proof in Proposition~\ref{prop:localization} below. For the last statement, we must show that the canonical map
	\begin{equation*}\label{eqn:LR-Kunneth}
	L_R(X\times Y) \to L_R(X) \times L_R(Y)
	\end{equation*}
	is an equivalence for all $X,Y\in\Pro(\scr S)$. Since both sides are $R$-local, it suffices to show that this map an $R$-homological equivalence. 
	By definition of $L_R$, the canonical map $C_*(X,R)\to C_*(L_RX,R)$ induces an isomorphism on homology pro-groups. Since $C_*(-,R)\colon \scr S\to \scr D(R)_{\geq 0}$ is a symmetric monoidal functor, we have a natural equivalence
	\[
	C_*(X\times Y,R) \simeq C_*(X,R) \otimes_R C_*(Y,R)
	\]
	in the $\infty$-category $\Pro(\scr D(R)_{\geq 0})$. Since $C_*(X,R)\to C_*(L_R(X),R)$ is a pro-homology isomorphism by definition of $L_R$, it remains to show that the tensor product in $\Pro(\scr D(R)_{\geq 0})$ preserves pro-homology isomorphisms. A morphism in $\Pro(\scr D(R)_{\geq 0})$ is a pro-homology isomorphism if and only if it induces an equivalence on $n$-truncations for all $n$, so the claim follows from the fact that the canonical map $M\otimes_RN \to \tau_{\leq n}M\otimes_R \tau_{\leq n} N$ is a $\tau_{\leq n}$-equivalence for all $M,N\in \scr D(R)_{\geq 0}$.
\end{proof}

The fact that $L_R$ preserves finite products is very useful and we will use it often. It implies in particular that $L_R$ preserves commutative monoids and commutes with the formation of strict symmetric powers. Here is another consequence:

\begin{corollary}\label{cor:distributive}
	Finite products distribute over finite colimits in $\Pro(\scr S)_R$.
\end{corollary}

\begin{proof}
	Finite colimits are universal in $\Pro(\scr S)$, \ie, are preserved by any base change (since pushouts and pullbacks can be computed levelwise). The result follows using that $L_R$ preserves finite products.
\end{proof}

\begin{remark}\label{rmk:protrunc}
	Let $\scr X$ be an $\infty$-topos and let $\scr X^\wedge$ be its hypercompletion. The geometric morphism $\scr X^\wedge\to\scr X$ induces an equivalence of pro-truncated shapes (since truncated objects are hypercomplete \cite[Lemma 6.5.2.9]{HTT}), and \textit{a fortiori} also of $R$-local and $R$-profinite shapes for any commutative ring $R$.
\end{remark}

\begin{remark}
	Let $l$ be a prime number. The Bockstein long exact sequences show that any $\Z/l$-cohomological equivalence is also a $\Z/l^n$-cohomological equivalence for all $n\geq 1$. In particular, if $\scr X$ is an $\infty$-topos, its $\Z/l$-profinite shape $L^{\Z/l}\Pi_\infty\scr X$ remembers the cohomology of $\scr X$ with $l$-adic coefficients.
\end{remark}

As shown in \cite[Proposition 7.3]{Isaksen:2005}, if $R$ is solid (e.g., $R=\Z/n$ for some integer $n$) and $X\in\scr S$ is connected, then $L_RX$ is the pro-truncation of the Bousfield–Kan $R$-tower of $X$ \cite[I, \S4]{BK}. 
It follows that the limit of the pro-space $L_R X$ is the Bousfield–Kan $R$-completion $R_\infty X$.

The existence of the localization functors $L_R$ and $L^R$ is a special case of a more general result which we now formulate.
If $\scr C$ is any locally small $\infty$-category, $\Pro(\scr C)^\op$ is the full subcategory of $\Fun(\scr C,\scr S)$ spanned by small filtered colimits of corepresentable functors \cite[Remark 5.3.5.9]{HTT}.  If $\scr K$ is any collection of objects of $\scr C$, a morphism $X\to Y$ in $\Pro(\scr C)$ is called a \emph{$\scr K$-equivalence} if it induces equivalences $\Map(Y,K)\simeq \Map(X,K)$ for every $K\in\scr K$, and an object $Z\in \Pro(\scr C)$ is called \emph{$\scr K$-local} if every $\scr K$-equivalence $X\to Y$ induces an equivalence $\Map(Y,Z)\simeq \Map(X,Z)$. Note that $\scr K$-equivalences are preserved by cofiltered limits, since $\scr K\subset\scr C$ and the objects of $\scr C$ are cocompact in $\Pro(\scr C)$. We denote by
\[
\Pro(\scr C)_{\scr K}\subset \Pro(\scr C)
\]
the full subcategory of $\scr K$-local objects.

\begin{proposition}\label{prop:localization}
	Let $\scr C$ be a presentable $\infty$-category and $\scr K$ a collection of objects of $\scr C$.
	Suppose that $\scr K$ is the essential image of an accessible functor.
	Then the inclusion $\Pro(\scr C)_{\scr K}\subset \Pro(\scr C)$ admits a left adjoint. Moreover, $\Pro(\scr C)_{\scr K}=\Pro(\widehat{\scr K})$ where $\widehat{\scr K}\subset\scr C$ is the closure of $\scr K$ under finite limits.
\end{proposition}

\begin{proof}
	With no assumptions on $\scr K$, there is an obvious inclusion $\Pro(\widehat{\scr K})\subset \Pro(\scr S)_{\scr K}$. 
	If $\scr K$ is small, then every functor $\widehat{\scr K}\to\scr S$ is a small colimit of corepresentables, so the inclusion $\Pro(\widehat{\scr K})\subset \Pro(\scr C)$ has a left adjoint $L$ given by restricting a functor $\scr C\to\scr S$ to $\widehat{\scr K}$. If $X\in \Pro(\scr C)_{\scr K}$, then the canonical map $X\to LX$ is a $\scr K$-equivalence between $\scr K$-local objects, hence it is an equivalence. This proves the proposition when $\scr K$ is small.
	
	The proof of the general case follows \cite[Proposition 6.10]{Isaksen:2005}.
	For any $X\in\Pro(\scr C)$, we must construct a $\scr K$-equivalence $X\to Y$ where $Y$ is $\scr K$-local.
	Suppose that $\scr K$ is the essential image of a functor $\phi\colon\scr L\to\scr C$ that preserves $\kappa$-filtered colimits, where $\scr C$ and $\scr L$ are $\kappa$-accessible. For $\lambda\gg\kappa$, let $\scr K^\lambda=\phi(\scr L^\lambda)$ where $\scr L^\lambda\subset\scr L$ is the subcategory of $\lambda$-compact objects.
	 Choose $\lambda_0\gg\kappa$ such that $X$ is a cofiltered limit of $\lambda_0$-compact objects of $\scr C$, and let $X\to Y_0$ be the $\scr K^{\lambda_0}$-localization of $X$. 
	 Inductively, choose $\lambda_n\gg \lambda_{n-1}$ such that $Y_{n-1}$ is a cofiltered limit of $\lambda_{n}$-compact objects, and let $X\to Y_n$ be the $\scr K^{\lambda_n}$-localization of $X$. 
 Finally, let $Y=\lim_n Y_n$. Then $Y$ is $\scr K$-local, and it remains to show that the induced map $X\to Y$ is a $\scr K$-equivalence.
	  
	 Let $K\in\scr K$ be the image of $L\in\scr L$.
	 For any $n\geq 0$, let $\scr L_n$ be the $\lambda_n$-filtered $\infty$-category $(\scr L^{\lambda_n})_{/L}$. Since $\phi$ preserves $\lambda_n$-filtered colimits, $K$ is the colimit of $\phi|\scr L_n$ for any $n$.
	 The conclusion now follows by evaluating the colimit
	 \[
	 \colim_{m,n}\colim_{A\in \scr L_n}\Map(Y_m, \phi(A))
	 \]
	 in two ways. Setting $m=n$, we have
	 \[
	 \colim_n \colim_{A\in \scr L_n}\Map(Y_n, \phi(A)) \simeq \colim_n \colim_{A\in \scr L_n}\Map(X, \phi(A))\simeq \Map(X,K),
	 \]
	 since $X\to Y_n$ is a $\scr K^{\lambda_n}$-equivalence and $X$ is a cofiltered limit of $\lambda_n$-compact objects for any $n$.
	 Setting $m=n-1$, we have
	 \[
	 \colim_n\colim_{A\in \scr L_n}\Map(Y_{n-1},\phi(A)) \simeq \colim_n \Map(Y_{n-1},K) \simeq \Map(Y,K),
	 \]
	 since $Y_{n-1}$ is a cofiltered limit of $\lambda_n$-compact objects.
	 This concludes the proof of the existence of the left adjoint.
	 By construction, $Y$ is in fact $\scr K^\lambda$-local for any large enough $\lambda\gg\kappa$, so we have also proved that
	 \begin{equation*}\label{eqn:lambda}
	 	\Pro(\scr C)_{\scr K}=\bigcup_{\lambda\gg\kappa}\Pro(\scr C)_{\scr K^\lambda}.
	 \end{equation*}
	 This implies that $\Pro(\scr C)_{\scr K}\subset \Pro(\widehat{\scr K})$, since we already know it when $\scr K$ is small.
\end{proof}

Proposition~\ref{prop:localization} applies in particular whenever $\scr K$ is small. For example, if $\scr K$ is the collection of Eilenberg–Mac Lane spaces $K(R,n)$ with $n\geq 0$, then $\Pro(\scr S)_{\scr K}=\Pro(\scr S)^R$. 
Proposition~\ref{prop:localization} also applies with $\scr K$ the collection of Eilenberg–Mac Lane spaces $K(M,n)$ for $M$ any $R$-module and $n\geq 0$, this being the image of a filtered-colimit-preserving functor from a countable disjoint union of copies of the category of $R$-modules. In this case, $\Pro(\scr S)_{\scr K}=\Pro(\scr S)_R$.

\begin{remark}\label{rmk:l-profinite}
	If $l$ is a prime number, the spaces that can be obtained from the Eilenberg–Mac Lane spaces $K(\Z/l,n)$ using finite limits are precisely the truncated spaces with finite $\pi_0$ and whose homotopy groups are finite $l$-groups.
	 Hence, $\Pro(\scr S)^{\Z/l}$ coincides with the $\infty$-category of $l$-profinite spaces studied in \cite[\S3]{DAG13}. 
\end{remark}

\begin{lemma}\label{lem:postnikov}
	Let $\scr C$ be a presentable $\infty$-category, let $(\scr K_\alpha)_\alpha$ be a small filtered diagram of collections of objects of $\scr C$ satisfying the assumption of Proposition~\ref{prop:localization}, and let $\scr K=\bigcup_\alpha\scr K_\alpha$. Then the localization functors induce an equivalence
	\[
	\Pro(\scr C)_{\scr K} \simeq \lim_\alpha \Pro(\scr C)_{\scr K_\alpha}.
	\]
\end{lemma}

\begin{proof}
	Note that $\scr K$ also satisfies the assumption of Proposition~\ref{prop:localization}, so that $\Pro(\scr C)_{\scr K}=\Pro(\widehat{\scr K})$.
	Since $\widehat{\scr K}=\colim_\alpha \widehat{\scr K}{}_\alpha$, we
	 have an equivalence 
	 \[\Fun(\widehat{\scr K},\scr S)\simeq \lim_\alpha \Fun(\widehat{\scr K}{}_\alpha,\scr S),
	 \]
	 which implies that the functor $\Pro(\scr C)_{\scr K} \to \lim_\alpha \Pro(\scr C)_{\scr K_\alpha}$ is fully faithful.
	 It remains to show that a functor $F\colon \widehat{\scr K}\to \scr S$ is a small filtered colimit of corepresentables if each of its restrictions $F|\widehat{\scr K}{}_\alpha$ is. This is true since $F\simeq \colim_\alpha F_\alpha$ where $F_\alpha$ is the left Kan extension of $F|\widehat{\scr K}{}_\alpha$ to $\widehat{\scr K}$.
\end{proof}

For $n\geq 0$, we define $\Pro(\scr S_{\leq n})_R$ to be the subcategory of $\scr K$-local objects in $\Pro(\scr S)$ where $\scr K$ is the collection of Eilenberg–Mac Lane spaces $K(M,i)$ with $M$ an $R$-module and $0\leq i\leq n$. 
Since $\scr K$ consists of $n$-truncated $R$-local objects, we have\footnote{Beware that this inclusion is strict for $n\geq 1$. In particular, $\Pro(\scr S_{\leq n})_R$ is not the subcategory of $n$-truncated objects in $\Pro(\scr S)_R$.}
\[
\Pro(\scr S_{\leq n})_R \subset \Pro(\scr S_{\leq n})\cap \Pro(\scr S)_R.
\]
By Proposition~\ref{prop:localization}, the inclusion $\Pro(\scr S_{\leq n})_R\subset \Pro(\scr S)$ admits a left adjoint $L_{R,\leq n}$. We define $\Pro(\scr S_{\leq n})^R$ and the localization functor $L^R_{\leq n}$ similarly (using only the $R$-module $R$).
By Lemma~\ref{lem:postnikov}, the localization functors $L_{R,\leq n}$ induce an equivalence
\begin{equation}\label{eqn:postnikov}
\Pro(\scr S)_R \simeq \lim_n \Pro(\scr S_{\leq n})_R
\end{equation}
and similarly for $\Pro(\scr S)^R$.

The following proposition shows that the localizations $L_R$ and $L^R$ agree in many cases of interest, partially answering \cite[Question 11.2]{Isaksen:2005}.

\begin{proposition}\label{prop:L_F=L^F}
	Let $F$ be a prime field and let $X$ be a pro-space whose $F$-homology pro-groups are pro-finite-dimensional vector spaces.
	 Then $L_FX$ is $F$-profinite. In other words, the canonical map $L_FX\to L^FX$ is an equivalence.
\end{proposition}

\begin{proof}
	First we claim that any $F$-profinite pro-space with profinite $\pi_0$ satisfies the given condition on $X$.
	Such a pro-space is a cofiltered limit of spaces with finite $\pi_0$ that are obtained from $K(F,n)$'s using finite limits.
	By \cite[Proposition 5.3]{BK}, each connected component of such a space is obtained from the point by a finite sequence of principal fibrations with fibers $K(F,n)$ with $n\geq 1$.
	Using Eilenberg–Moore \cite[Corollary 1.1.10]{DAG13}, it thus suffices to show that $H^m(K(F,n),F)$ is finite-dimensional for every $m\geq 0$ and $n\geq 1$, which is a well-known computation.
	Thus, both $X$ and $L^FX$ have pro-finite-dimensional $F$-homology pro-groups. 
	It follows that the canonical map $X\to L^FX$ induces an isomorphism on $F$-cohomology ind-groups, whence on $F$-homology pro-groups. 
\end{proof}

\begin{remark}\label{rmk:finitehomology}
	It is clear that the class of pro-spaces $X$ satisfying the hypothesis of Proposition~\ref{prop:L_F=L^F} is preserved by $L_F$, retracts, finite products, and finite colimits (it suffices to verify the latter for pushouts).
\end{remark}

\begin{proposition}\label{prop:Kunneth2}
	Let $F$ be a prime field and let $X$ be a pro-space whose $F$-homology pro-groups in degrees $\leq n$ are pro-finite-dimensional vector spaces. Then the canonical map $L_{F,\leq n}X\to L^F_{\leq n}X$ is an equivalence. Furthermore, $L_{F,\leq n}$ preserves finite products of such pro-spaces.
\end{proposition}

\begin{proof}
	As in the proof of Proposition~\ref{prop:L_F=L^F}, $L^F_{\leq n}X$ has pro-finite-dimensional $F$-homology pro-groups, hence the canonical map $X\to L^F_{\leq n}X$ induces an isomorphism on $F$-cohomology ind-groups in degrees $\leq n$. Since these groups are finite-dimensional $F$-vector spaces, this remains true for cohomology with coefficients in any $F$-vector space, which implies the first statement. For the second statement, we must show that the canonical map
	\[
	X\times Y \to L_{\leq n}^FX \times L_{\leq n}^F Y
	\]
	induces an isomorphism on $F$-cohomology in degrees $\leq n$. 
	Since $F$ is a field, we have an isomorphism of graded pro-vector spaces $H_*(X\times Y,F)\simeq H_*(X,F)\otimes_F H_*(Y,F)$, which are pro-finite-dimensional in degrees $\leq n$.
	Dualizing and taking the colimit, we find that
	\[
	H^m(X\times Y, F) \simeq \bigoplus_{r+s=m}H^r(X,F) \otimes H^s(Y,F)
	\]
	for any $m\leq n$. The same formula applies to $L_{\leq n}^FX$ and $L_{\leq n}^FY$. By definition of $L_{\leq n}^F$, we have $H^m(X,F)\simeq H^m(L_{\leq n}^FX,F)$ for all $m\leq n$ and similarly for $Y$, so we are done.
\end{proof}

\section{The h and qfh topologies} 
\label{sec:the_qfh_topology}

Let $X$ be a Noetherian scheme. An \emph{h covering} of $X$ is a finite family $\{U_i\to X\}$ of morphisms of finite type such that the induced morphism $\coprod_i U_i\to X$ is universally submersive (a morphism of schemes $f\colon Y\to X$ is submersive if it is surjective and if the underlying topological space of $X$ has the quotient topology). If in addition each $U_i\to X$ is quasi-finite, it is a \emph{qfh covering}. These notions of coverings define pretopologies on Noetherian schemes which we denote by h and qfh, respectively. The h and qfh topologies are both finer than the fppf topology, and they are \emph{not} subcanonical.

\begin{proposition}
	\label{prop:qfh=et}
	Let $X$ be a Noetherian scheme.
	\begin{enumerate}
		\item The canonical map $\Pi_\infty^\qfh X\to \Pi_\infty^\et X$ induces an isomorphism in cohomology with any local system of abelian coefficients. In particular, for any commutative ring $R$,
		\[L_{R}\Pi_\infty^\qfh X\simeq L_{R}\Pi_\infty^\et X.\]
		\item If $X$ is excellent, the canonical map $\Pi_\infty^\h X\to\Pi_\infty^\qfh X$ induces an isomorphism in cohomology with any local system of \emph{torsion} abelian coefficients. In particular, for any \emph{torsion} commutative ring $R$,
		\[L_{R}\Pi_\infty^\h X\simeq L_{R}\Pi_\infty^\qfh X.\]
	\end{enumerate}
\end{proposition}

\begin{proof}
Recall that the cohomology of $\Pi_\infty\scr X$ with coefficients in a local system coincides with the cohomology of $\scr X$ with coefficients in the associated locally constant sheaf (see \S\ref{sec:homotopy_types_of_schemes}).
	The first statements are thus translations of \cite[Theorem 3.4.4]{HS1} and \cite[Theorem 3.4.5]{HS1}, respectively (the excellence of $X$ is a standing assumption in \emph{loc.\ cit.}, but it is not used in the proof of (1); see also \cite[\S10]{Suslin:1996} for self-contained proofs). The statements about the $R$-local shapes follow immediately, since $L_R$ inverts morphisms that induce isomorphisms in cohomology with coefficients in any $R$-module (see \S\ref{sec:homological_localizations}).
\end{proof}

\begin{remark}\label{rmk:Jardine}
	Voevodsky's proof also shows that $H^1_\et(X,G)\simeq H^1_\qfh(X,G)$ for any locally constant étale sheaf of groups $G$. It follows from \cite[Proposition 18.4]{Isaksen:2001} (and Remark~\ref{rmk:Isaksen}) that $\Pi_\infty^\qfh X\to \Pi_\infty^\et X$ is in fact a $\tau_{<\infty}$-equivalence.
\end{remark}

For $\scr C$ a small $\infty$-category, we denote by $r\colon \scr C\into \PSh(\scr C)$ the Yoneda embedding, and if $\tau$ is a topology on $\scr C$, we denote by $r_\tau =a_\tau r$ the $\tau$-sheafified Yoneda embedding.

\begin{lemma}\label{lem:compact}
	Let $S$ be a Noetherian scheme and let $\tau\in\{\h,\qfh,\et\}$. Then, for any $n\geq 0$, the image of the Yoneda functor $r_\tau\colon \Sch_S^\ft\to \Shv_\tau(\scr \Sch_S^\ft)_{\leq n}$ consists of compact objects.
\end{lemma}

\begin{proof}
	The category $\Sch_S^\ft$ has finite limits and the topology $\tau$ is finitary, and so the $\infty$-topos $\Shv_\tau(\Sch_S^\ft)$ is locally coherent and coherent. The result now follows from \cite[Corollary 2.3.10(1)]{DAG13}.
\end{proof}

The advantage of the qfh topology over the étale topology is that it can often cover singular schemes by smooth schemes. Let us make this explicit in the case of quotients of smooth schemes by finite group actions. We first recall the classical existence result for such quotients.

A groupoid scheme $X_\bullet$ is a simplicial scheme such that, for every scheme $Y$, $\Hom(Y,X_\bullet)$ is a groupoid. If $P$ is any property of morphisms of schemes that is stable under base change, we say that $X_\bullet$ has property $P$ if every face map in $X_\bullet$ has property $P$ (of course, it suffices that $d_0\colon X_1\to X_0$ have property $P$).

\begin{lemma}\label{lem:quotient}
	Let $S$ be a scheme and let $X_\bullet$ be a finite and locally free groupoid scheme over $S$. Suppose that for any $x\in X_0$, $d_1(d_0^{-1}(x))$ is contained in an affine open subset of $X_0$ (for example, $X_0$ is quasi-projective over $S$). Then $X_\bullet$ has a colimiting cone $p\colon X_\bullet\to Y$ in the category of $S$-schemes. Moreover,
	\begin{enumerate}
		\item $p$ is integral and surjective, and in particular universally submersive;
		\item the canonical morphism $X_\bullet\to \cosk_0(p)$ is degreewise surjective;
		\item if $S$ is locally Noetherian and $X_0$ is of finite type over $S$, then $Y$ is of finite type over $S$.
	\end{enumerate}
\end{lemma}

\begin{proof}
	The claim in parentheses follows from~\cite[Corollaire 4.5.4]{EGA2} and the definition of quasi-projective morphism~\cite[Définition 5.3.1]{EGA2}. An integral and surjective morphism is universally submersive because integral morphisms are closed \cite[Proposition 6.1.10]{EGA2}.
	The existence of $p$ which is integral and surjective and (2) are proved in \cite[III, \S2, 3.2]{Demazure:1970} or \cite[V, Théorème 4.1]{SGA3}. Part (3) is proved in \cite[V, Lemme 6.1(ii)]{SGA3}. 
\end{proof}

\begin{proposition}\label{prop:quotient}
	Let $S$ be a Noetherian scheme and $X_\bullet$ a groupoid scheme of finite type over $S$ as in Lemma~\ref{lem:quotient} with colimit $Y$. Then $r_\qfh X_\bullet\to r_\qfh Y$ is a colimiting cone in $\Shv_\qfh(\Sch_S^\ft)$.
\end{proposition}

\begin{proof}
	By Lemma~\ref{lem:quotient}, the colimiting cone $X_\bullet\to Y$ is a qfh hypercover, and it is $2$-coskeletal since $X_\bullet$ is a groupoid scheme. Hence, if $\scr F$ is a qfh sheaf, we have $\scr F(Y)\simeq \lim\scr F(X_\bullet)$ by \cite[Lemma 6.5.3.9]{HTT}. 
\end{proof}

\begin{corollary}\label{cor:quotient2}
	Let $S$ be a Noetherian scheme and $X$ a quasi-projective $S$-scheme.
	Then the Yoneda functor $r_\qfh\colon\Sch_S^\ft\to\Shv_\qfh(\Sch_S^\ft)$ preserves strict symmetric powers of $X$, \ie, it sends the schematic symmetric power $\Sym^dX$ to the sheafy symmetric power $\Sym^dr_\qfh X$.
\end{corollary}

\begin{proof}
	Since the strict symmetric power is the usual symmetric power in a $1$-category, we have $r_\qfh\Sym^d X\simeq \Sym^dr_\qfh X$ in $\Shv_\qfh(\Sch_S^\ft)_{\leq 0}$ by Proposition~\ref{prop:quotient}, whence in $\Shv_\qfh(\Sch_S^\ft)$ by Lemma~\ref{lem:symmtopos}.
\end{proof}

\begin{proposition}\label{prop:dejong}
	Let $k$ be a perfect field and $i\colon \scr C\into \Sch_k^\ft$ the inclusion of a full subcategory such that every smooth $k$-scheme is Zariski-locally in $\scr C$. Let $i^*\colon\PSh(\Sch_k^\ft)\to\PSh(\scr C)$ be the restriction functor and let $i_!$ be its left adjoint.
	For every $\scr F\in \PSh(\Sch_k^\ft)$, the counit morphism $i_!i^*\scr F \to \scr F$ is an $a_\h^\wedge$-equivalence. 
\end{proposition}

\begin{proof}
	By a well-known theorem of de Jong \cite[Theorem 4.1]{DeJong}, every scheme of finite type over a perfect field $k$ is h-locally smooth, hence is h-locally in $\scr C$. It follows that there is an induced h topology on $\scr C$ whose covering sieves are the restrictions of h-covering sieves in $\Sch_k^\ft$, or equivalently those sieves that generate an h-covering sieve in $\Sch_k^\ft$. By the comparison lemma \cite[III, Théorème 4.1]{SGA4}, the restriction functor $i^*$ and its right adjoint $i_*$ restrict to an equivalence between the subcategories of h sheaves of sets, hence between the $\infty$-categories of hypercomplete h sheaves (since they are the hypercompletions of the associated $1$-localic $\infty$-topoi).
\end{proof}

\section{The étale homotopy type of symmetric powers} 
\label{sec:homotopy_type_of_symmetric_powers}

\begin{proposition}\label{prop:etalekunneth}
	Let $k$ be a separably closed field, $l\neq\car k$ a prime number, and $X$ and $Y$ schemes of finite type over $k$. Let $\tau\in\{\h,\qfh,\et\}$. Then
	\[L_{\Z/l}\Pi_\infty^\tau(X\times_kY)\simeq L_{\Z/l}\Pi_\infty^\tau X\times L_{\Z/l}\Pi_\infty^\tau Y.\]
\end{proposition}

\begin{proof}
	By Proposition~\ref{prop:qfh=et}, it suffices to prove the lemma for $\tau=\et$. Since $L_{\Z/l}$ preserves finite products and $X_\et$ and $X_\et^\wedge$ have the same $\Z/l$-local shape (see Remark~\ref{rmk:protrunc}), it suffices to show that the canonical map
\begin{equation}\label{eqn:kunneth}
	\Pi_\infty(X\times_kY)_\et^\wedge\to\Pi_\infty X_\et^\wedge\times\Pi_\infty Y_\et^\wedge
\end{equation}
is a $\Z/l$-homological equivalence or, equivalently, that it induces an isomorphism in cohomology with coefficients in any $\Z/l$-module $M$. Both sides of~\eqref{eqn:kunneth} are corepresented by cofiltered diagrams of simplicial sets having finitely many simplices in each degree (by Remark~\ref{rmk:qwe} and the fact that any étale hypercovering of a Noetherian scheme is refined by one that is degreewise Noetherian). If $K$ is any such pro-space, $C_\ast(K,\Z/l)$ is a cofiltered diagram of degreewise finite chain complexes of vector spaces. On the one hand, this implies
	\[C^\ast(K,M)\simeq C^\ast(K,\Z/l)\tens M,\]
	so we may assume that $M=\Z/l$. On the other hand, it implies that the Künneth map
	\[H^\ast(\Pi_\infty X_\et^\wedge,\Z/l)\tens H^\ast(\Pi_\infty Y_\et^\wedge,\Z/l)\to H^\ast(\Pi_\infty X_\et^\wedge\times\Pi_\infty Y_\et^\wedge,\Z/l)\]
	is an isomorphism. The composition of this isomorphism with the map induced by~\eqref{eqn:kunneth} in cohomology is the canonical map
	\[H^\ast_\et(X,\Z/l)\tens H^\ast_\et(Y,\Z/l)\to H^\ast_\et(X\times_kY,\Z/l),\]
	which is also an isomorphism by~\cite[Th.~finitude, Corollaire 1.11]{SGA4.5}.
\end{proof}

\begin{remark}\label{rmk:finitevs}
	Let $X$ be a Noetherian scheme and let $\tau\in\{\h,\qfh,\et\}$. We observed in the proof of Proposition~\ref{prop:etalekunneth} that the pro-space $\Pi_\infty (X_\tau^\wedge)$ is the limit of a cofiltered diagram of spaces whose integral homology groups are finitely generated. It follows from Proposition~\ref{prop:L_F=L^F} that $L_F\Pi_\infty^\tau X$ is $F$-profinite for any prime field $F$, which answers \cite[Question 11.3]{Isaksen:2005} quite generally.
\end{remark}

Now let $\tau$ and $\sigma$ be pretopologies on $\Sch_S$ with $\tau$ finer than $\sigma$, and let $\scr C$ be a small full subcategory of $\Sch_S$ closed under $\sigma$-coverings (but not necessarily under $\tau$-coverings). 
Then the functor $\Pi_\infty^\tau\colon\scr C\to\Pro(\scr S)$ takes values in a cocomplete $\infty$-category and is a $\sigma$-cosheaf according to Lemma~\ref{lem:descent}, so it lifts uniquely to a left adjoint functor
\begin{tikzcentered}
	\diagram{\scr C & \Pro(\scr S)\rlap. \\ \Shv_\sigma(\scr C) & \\};
	\arrows (11-) edge node[above]{$\Pi_\infty^\tau$} (-12) (11) edge (21) (21) edge[dashed] node[below right]{$\LPi_\infty^\tau$} (12);
\end{tikzcentered}

\begin{remark}\label{rmk:pi_infty}
	If $\scr C$ is such that $\scr C_{/X}$ contains the small $\tau$-site of $X$ for any $X\in\scr C$, then $\LPi_\infty^\tau$ is simply the composition
	\begin{tikzcentered}
		\def\colsep{1em}
		\diagram{\Shv_\sigma(\scr C) & & \Shv_\tau(\scr C) & \TopR & & \Pro(\scr S)\rlap, \\};
		\arrows (11-) edge node[above]{$a_\tau$} (-13) (13-) edge[right hook->] (-14) (14-) edge node[above]{$\Pi_\infty$} (-16);
	\end{tikzcentered}
	where for $\scr X$ an $\infty$-topos the inclusion $\scr X\into\TopR$ is $X\mapsto \scr X_{/X}$. Indeed, this composition preserves colimits (by \cite[Proposition 6.3.5.14]{HTT}), and it restricts to $\Pi_\infty^\tau$ on $\scr C$ (\cf{} Remark~\ref{rmk:site}). However, the reader should be warned that we will use $\LPi_\infty^\tau$ in situations where this hypothesis on $\scr C$ is not satisfied.
\end{remark}

\begin{remark}
	The extension $\LPi_\infty^\tau$ involves taking infinite colimits in $\Pro(\scr S)$, which are somewhat ill-behaved (they are not universal, for example). 
	As we will see in \S\ref{sec:EML}, it is sometimes advantageous to consider a variant of $\LPi_\infty^\tau$ taking values in ind-pro-spaces.
\end{remark}

\begin{theorem}\label{thm:main}
	Let $k$ be a separably closed field, $l\neq\car k$ a prime number, and $X$ a quasi-projective scheme over $k$. Let $\tau\in\{\h,\qfh,\et\}$. Then for any $d\geq 0$ there is a canonical equivalence
	\[L_{\Z/l}\Pi_\infty^\tau(\Sym^dX)\simeq L_{\Z/l} \Sym^d\Pi_\infty^\tau(X).\]
\end{theorem}

\begin{proof}
	Let $\scr C$ be the category of quasi-projective schemes over $k$. By Corollary~\ref{cor:quotient2}, the representable sheaf functor $r_\qfh\colon \scr C\to\Shv_\qfh(\scr C)$ preserves strict symmetric powers. Using Proposition~\ref{prop:etalekunneth} and the fact that $L_{\Z/l}\LPi_\infty^\qfh$ preserves colimits, we deduce that $L_{\Z/l}\Pi_\infty^\qfh$ preserves strict symmetric powers on $\scr C$. For $\tau\in\{\h,\qfh,\et\}$, we get
	\[L_{\Z/l}\Pi_\infty^\tau(\Sym^dX)\simeq L_{\Z/l}\Pi_\infty^\qfh(\Sym^d X)\simeq \Sym^d L_{\Z/l}\Pi_\infty^\qfh(X)\simeq \Sym^d L_{\Z/l}\Pi_\infty^\tau(X),\]
	the first and last equivalences being from Proposition~\ref{prop:qfh=et}. The functor $L_{\Z/l}$ itself also preserves finite products (Proposition~\ref{prop:finiteprod}) and hence strict symmetric powers, so we are done.
\end{proof}

\begin{remark}
	It is possible to define a natural map
	\[\Pi_\infty ((\Sym^d X)_\tau^\wedge)\to\Sym^d\Pi_\infty (X_\tau^\wedge)\]
	in $\Pro(\scr S)$ inducing the equivalence of Theorem~\ref{thm:main}.
	It suffices to make the square
	\begin{tikzequation}\label{eqn:asd}
		\diagram{\Pi_\infty((X^d)_\tau^\wedge) & \Pi_\infty (X_\tau^\wedge)^d \\ \Pi_\infty((\Sym^dX)_\tau^\wedge) & \Sym^d\Pi_\infty (X_\tau^\wedge) \\};
		\arrows (11-) edge (-12) (11) edge (21) (12) edge (22) (21-) edge[dashed] (-22);
	\end{tikzequation}
	commute. Using the model for the $\tau$-homotopy type discussed in Remark~\ref{rmk:qwe} and the commutativity of the functor of connected components with symmetric powers, the task to accomplish is the following: associate to any $\tau$-hypercover $U_\bullet\to X$ a $\tau$-hypercover $V_\bullet\to\Sym^d X$ refining $\Sym^dU_\bullet\to\Sym^d X$ and such that $V_\bullet\times_{\Sym^d X}X^d\to X^d$ refines $U_\bullet^d\to X^d$, in a simplicially enriched functorial way (\ie, we must define a simplicial functor $\HC_\tau(X)\to\HC_\tau(\Sym^d X)$ and the refinements must be natural). If $\tau=\h$ or $\tau=\qfh$, $\Sym^dU_\bullet \to\Sym^d X$ is itself a $\tau$-hypercover and we are done, but things get more complicated for $\tau=\et$ as symmetric powers of étale maps are not étale anymore.
	
	We refer to~\cite[\S4.5]{Kollar:1997} and~\cite[\S3]{Rydh:2012} for more details on the following ideas. Given a finite group $G$ and quasi-projective $G$-schemes $U$ and $X$, a map $f\colon U\to X$ is $G$-equivariant if and only if it admits descent data for the action groupoid of $G$ on $X$. The map $f$ is \emph{fixed-point reflecting} if it admits descent data for the Čech groupoid of the quotient map $X\to X/G$ (this condition can be expressed more explicitly using the fact that $G\times X\to X\times_{X/G}X$ is faithfully flat: $f$ is fixed-point reflecting if and only if it is $G$-equivariant and induces a fiberwise isomorphism between the stabilizer schemes). Since étale morphisms descend effectively along universally open surjective morphisms \cite[Theorem 5.19]{Rydh:2010}, such as $X\to X/G$, the condition that $f$ reflects fixed points is equivalent to the induced map $U/G\to X/G$ being étale and the square
	\begin{tikzcentered}
		\diagram{U & X \\ U/G & X/G \\};
		\arrows (11-) edge (-12) (11) edge (21) (12) edge (22) (21-) edge (-22);
	\end{tikzcentered}
	being cartesian. If $f$ is $G$-equivariant, there exists a largest $G$-equivariant open subset $\fpr(f)\subset U$ on which $f$ is fixed-point reflecting. Moreover, if $f\colon U\to X$ is an étale cover, the restriction of $f^d$ to $\fpr(f^d)$ is still surjective. Now given an étale hypercover $U_\bullet \to X$, we can define an étale hypercover $V_\bullet\to\Sym^d X$ refining $\Sym^d U_\bullet\to\Sym^d X$ as follows. Let $W_0\subset U_0^d$ be the locus where $U_0^d\to X^d$ reflects fixed points. If $W_\bullet$ has been defined up to level $n-1$, define $W_n$ by the cartesian square
	\begin{tikzcentered}
		\def\colsep{0em}
		\diagram{W_n & \fpr(U_n^d\to (\cosk_{n-1}U_\bullet^d)_n) \\
		(\cosk_{n-1}W_\bullet)_n & (\cosk_{n-1}U_\bullet^d)_n \\};
		\arrows (11-) edge (-12) (11) edge (21) (21-) edge (-22) (12) edge (22);
	\end{tikzcentered}
	in which the vertical maps are $\Sigma_d$-equivariant fixed-point reflecting étale covers (because fixed-point reflecting morphisms are stable under base change).  Finally, let $V_n=W_n/\Sigma_d$. It is then easy to prove that $V_\bullet\to \Sym^d X$ is an étale hypercover with the desired functoriality.
	
	Using the commutativity of~\eqref{eqn:asd}, one can also show that the map induced by $\Pi_\infty ((\Sym^d X)_\et^\wedge)\to\Sym^d\Pi_\infty (X_\et^\wedge)$ in cohomology with coefficients in a $\Z/l$-module coincides with the symmetric Künneth map defined in~\cite[XVII, (5.5.17.2)]{SGA4}. Thus, for $X$ proper, it is possible to deduce Theorem~\ref{thm:main} from~\cite[XVII, Théorème 5.5.21]{SGA4}.
\end{remark}

\section{\texorpdfstring{$\A^1$}{A1}-localization} 
\label{sec:section_name}

Let $S$ be a quasi-compact quasi-separated scheme and $\scr C$ a full subcategory of $\Sch_S$ such that
\begin{enumerate}
	\item objects of $\scr C$ are of finite presentation over $S$;
	\item $S\in\scr C$ and $\A^1_S\in\scr C$;
	\item if $X\in\scr C$ and $U\to X$ is étale, separated, and of finite presentation, then $U\in\scr C$;
	\item $\scr C$ is closed under finite products and finite coproducts.
\end{enumerate}
Following~\cite[\S 0]{MEMS}, we call such a category $\scr C$ \emph{admissible}. 
Note that every smooth $S$-scheme admits an open covering by schemes in $\scr C$.
Let $\Shv_\Nis(\scr C)$ denote the $\infty$-topos of sheaves of spaces on $\scr C$ for the Nisnevich topology, and let $\Shv_{\Nis}(\scr C)_{\A^1}\subset \Shv_\Nis(\scr C)$ be the full subcategory of $\A^1$-invariant Nisnevich sheaves.
We shall denote by
\[
L_{\Nis,\A^1}\colon \PSh(\scr C)\to \Shv_{\Nis}(\scr C)_{\A^1}
\]
the left adjoint to the inclusion.

From now on we fix a prime number $l$ different from the residual characteristics of $S$.
In \S\ref{sec:homotopy_types_of_schemes}, we defined the functor
\[\Pi_\infty^\et\colon\scr C\to \Pro(\scr S),\]
and we observed in \S\ref{sec:homotopy_type_of_symmetric_powers} that it lifts to a left adjoint functor
\[\LPi_\infty^\et\colon\Shv_\Nis(\scr C)\to\Pro(\scr S).\]
By~\cite[VII, Corollaire~1.2]{SGA5}, the composition 
\begin{tikzcentered}
	\diagram{\Shv_\Nis(\scr C) & \Pro(\scr S) & \Pro(\scr S)_{\Z/l} \\};
	\arrows (11-) edge node[above]{$\LPi_\infty^\et$} (-12) (12-) edge node[above]{$L_{\Z/l}$} (-13);
\end{tikzcentered}
 sends any morphism $\A^1\times X\to X$  in $\scr C$ to an equivalence and therefore factors through the $\A^1$-localization functor $L_{\Nis,\A^1}$. That is, there is a commutative square of left adjoint functors
\begin{tikzmath}
	\def\colsep{3.5em}
	\diagram{\Shv_\Nis(\scr C) & \Pro(\scr S) \\ \Shv_{\Nis}(\scr C)_{\A^1} & \Pro(\scr S)_{\Z/l}\rlap. \\};
	\arrows (11-) edge node[above]{$\LPi_\infty^\et$} (-12) (11) edge node[left]{$L_{\Nis,\A^1}$} (21) (21-) edge[dashed] node[above]{$\Et_l$} (-22) (12) edge node[right]{$L_{\Z/l}$} (22);
\end{tikzmath}
The functor $\Et_l$ is called the \emph{$\Z/l$-local étale homotopy type} functor.
Note that if $S$ is Noetherian (\resp{} Noetherian and excellent), we could also use $\LPi_\infty^\qfh$ (\resp{} $\LPi_\infty^\h$) instead of $\LPi_\infty^\et$ in the above diagram, according to Proposition~\ref{prop:qfh=et}.

\begin{remark}\label{rmk:profinite-realization}
	The $\Z/l$-profinite completion $L^{\Z/l}\Et_l$ is the $\infty$-categorical incarnation of the étale realization functor defined by Isaksen in~\cite{Isaksen:2004} as a left Quillen functor, but our results do not require this stronger completion. Note that $\Et_l X$ is already $\Z/l$-profinite if $S$ is Noetherian and $X\in (\Shv_\Nis(\scr C)_{\A^1})^\omega$, by Remarks~\ref{rmk:finitevs} and~\ref{rmk:finitehomology}.
\end{remark}

We now assume that $S=\Spec k$ where $k$ is a separably closed field. 

\begin{lemma}\label{lem:Etdistributive}
	The restriction of $\Et_l$ to the subcategory of compact objects $(\Shv_\Nis(\scr C)_{\A^1})^\omega$ preserves finite products.
\end{lemma}

\begin{proof}
	By Proposition~\ref{prop:etalekunneth}, the functor $L_{\Z/l}\Pi_\infty^\et$ preserves finite products on $\scr C$. Since the functor $L_{\Nis,\A^1} r\colon \scr C\to \Shv_\Nis(\scr C)_{\A^1}$ also preserves finite products, the restriction of $\Et_l$ to the image of $\scr C$ preserves finite products.
	Finally, since $(\Shv_\Nis(\scr C)_{\A^1})^\omega$ is the closure of the image of $\scr C$ under finite colimits and retracts, the result follows from Corollary~\ref{cor:distributive}.
\end{proof}

Let $p\geq q\geq 0$. We define the $\Z/l$-local mixed spheres $S^{p,q}_l\in\Pro(\scr S)_{\Z/l,*}$ by
\[S^{1,0}_l=L_{\Z/l}S^1=K(\Z_l,1),\quad S^{1,1}_l=K(T_l\mu,1),\quad S^{p,q}_l=(S^{1,0}_l)^{\wedge p-q}\wedge (S^{1,1}_l)^{\wedge q},\]
where $\mu$ is the group of roots of unity in $k$ and $T_l\mu=\lim_n \mu_{l^n}$ is its $l$-adic Tate module. Here we regard $\Z_l$ and $T_l\mu$ as pro-groups.
Of course, $S^{p,q}_l\simeq L_{\Z/l}S^p$, but if $q>0$ this equivalence depends on infinitely many noncanonical choices (\viz, an isomorphism $\Z_l\simeq T_l\mu$). By Proposition~\ref{prop:L_F=L^F}, $S^{p,q}_l$ is $\Z/l$-profinite.

Note that the functor $\Et_l$ preserves pointed objects, since $\Pi_\infty^\et(\Spec k)\simeq\ast$.

\begin{proposition}\label{prop:spheres}
	Let $p\geq q\geq 0$. Then $\Et_lS^{p,q}\simeq S^{p,q}_l$.
\end{proposition}

\begin{proof}
	This is obvious if $q=0$. By Lemma~\ref{lem:Etdistributive}, it remains to treat the case $p=q=1$. 
	The étale $\mu_{l^n}$-torsor $l^n\colon \G_m\to \G_m$ is classified by a morphism $\Pi_\infty^\et\G_m\to K(\mu_{l^n},1)$. In the limit over $n\geq 0$, we obtain a morphism of pro-spaces $\phi\colon \Pi_\infty^\et\G_m \to K(T_l\mu,1)$.
	We claim that $\phi$ is a $\Z/l$-homological equivalence, i.e., it induces an isomorphism in cohomology with coefficients in any $\Z/l$-module $M$.
	By~\cite[VII, Proposition 1.3(i)(c)]{SGA5}, we have
	\[H^i_\et(\G_m,M)=\begin{cases} M & \text{if $i=0$,} \\ \Hom(\mu_l,M) & \text{if $i=1$,} \\ 0 & \text{if $i\geq 2$.}\end{cases} \]
	In fact, this computation shows that the morphism $\Pi_\infty^\et\G_m\to K(\mu_l,1)$ induces an isomorphism on $H^i(-,M)$ for $i\leq 1$. The same is true for the projection $K(T_l\mu,1)\to K(\mu_l,1)$, hence also for $\phi$. Since both the source and the target of $\phi$ have vanishing cohomology in degrees $\geq 2$,
	this completes the proof.
\end{proof}

\section{Group completion and strictly commutative monoids} 
\label{sec:group_completions}

Let $\scr C$ be an $\infty$-category with finite products. Recall from~\cite[\S2.4.2]{HA} that a commutative monoid in $\scr C$ is a functor $M\colon \Fin_\ast\to\scr C$ such that for all $n\geq 0$ the canonical map $M(\langle n\rangle)\to M(\langle 1\rangle)^n$ is an equivalence. We let $\CMon(\scr C)$ denote the full subcategory of $\Fun(\Fin_\ast,\scr C)$ spanned by the commutative monoids.

A commutative monoid $M$ in $\scr C$ has an underlying simplicial object, namely its restriction along the functor $\Cut\colon\Delta^\op\to\Fin_\ast$ sending $[n]$ to the finite set of cuts of $[n]$ pointed at the trivial cut, which can be identified with $\langle n\rangle$. 
The commutative monoid $M$ is called \emph{grouplike} if its underlying simplicial object is a groupoid object in the sense of~\cite[Definition 6.1.2.7]{HTT}. This is equivalent to requiring both shearing maps $M\times M\to M\times M$ to be equivalences.
We denote by $\CMon^\gr(\scr C)\subset \CMon(\scr C)$ the full subcategory of grouplike objects.

If $f\colon\scr C\to\scr D$ preserves finite products (and $\scr C$ and $\scr D$ admit finite products), then it induces a functor $\CMon(\scr C)\to\CMon(\scr D)$ by postcomposition; this functor clearly preserves grouplike objects and hence restricts to a functor $\CMon^\gp(\scr C)\to\CMon^\gp(\scr D)$. We will continue to use $f$ to denote either induced functor.

\begin{lemma}\label{lem:qr5w}
	Suppose that $f\colon\scr C\to\scr D$ preserves finite products and has a right adjoint $g$. Then the functors $\CMon(\scr C)\to\CMon(\scr D)$ and $\CMon^\gr(\scr C)\to\CMon^\gr(\scr D)$ induced by $f$ are left adjoint to the corresponding functors induced by $g$.
\end{lemma}

\begin{proof}
	The functors $f$ and $g$ induce adjoint functors between $\infty$-categories of $\Fin_*$-diagrams, and it remains to observe that they both preserve the full subcategory of (grouplike) commutative monoids.
\end{proof}

\begin{definition}\label{def:distributive}
	An $\infty$-category $\scr C$ is \emph{distributive} if it is presentable and if finite products in $\scr C$ distribute over colimits. A functor $f\colon\scr C\to\scr D$ between distributive $\infty$-categories is \emph{distributive} if it preserves colimits and finite products.
\end{definition}

For example, for any $\infty$-topos $\scr X$, the truncation functors $\tau_{\leq n}\colon\scr X\to\scr X_{\leq n}$ are distributive, and for any admissible category $\scr C\subset\Sch_S$ and any topology $\tau$ on $\scr C$, the localization functor $L_{\tau,\A^1}\colon\PSh(\scr C)\to\Shv_\tau(\scr C)_{\A^1}$ is distributive. 

If $\scr C$ is distributive, then the $\infty$-category $\CMon(\scr C)$ is presentable by~\cite[Corollary 3.2.3.5]{HA}, and the subcategory $\CMon^\gp(\scr C)$ is strongly reflective since it is accessible (by~\cite[Proposition 6.1.2.9]{HTT} and~\cite[Proposition 5.4.6.6]{HTT}) and is clearly closed under limits and sifted colimits. That is, there exists a \emph{group completion} functor
\[\CMon(\scr C)\to \CMon^\gp(\scr C), \quad M\mapsto M^\gp,\]
which exhibits $\CMon^\gp(\scr C)$ as an accessible localization of $\CMon(\scr C)$.\footnote{In fact, $\CMon(\scr C)$ and $\CMon^\gr(\scr C)$ are presentable whenever $\scr C$ is, and the group completion functor exists in that generality, see \cite[\S 4]{GGN}. However, these constructions may not behave as expected if $\scr C$ is not distributive; for example, the forgetful functor $\CMon(\scr C)\to \scr C$ and the inclusion $\CMon^\gp(\scr C)\subset \CMon(\scr C)$ need not preserve sifted colimits.}

\begin{lemma}\label{lem:grp}
	Let $f\colon\scr C\to\scr D$ be a distributive functor and let $M\in \CMon(\scr C)$. 
	Then $f(M^\gp)\simeq f(M)^\gp$.
\end{lemma}

\begin{proof}
	By Lemma~\ref{lem:qr5w}, the square
	\begin{tikzcentered}
		\diagram{\CMon(\scr C) & \CMon(\scr D) \\ \CMon^\gp(\scr C) & \CMon^\gp(\scr D)\\};
		\arrows (11-) edge node[above]{$f$} (-12) (11) edge node[left]{$\gp$} (21) (21-) edge node[above]{$f$} (-22) (12) edge node[right]{$\gp$} (22);
	\end{tikzcentered}
	has a commutative right adjoint and hence is commutative.
\end{proof}

\begin{remark}
	If $\scr X$ is an $\infty$-topos, group completion of commutative monoids in $\scr X$ preserves $0$-truncated objects. As in the proof of Lemma~\ref{lem:symmtopos}, it suffices to prove this for $\scr X=\scr S$, where it follows from the McDuff–Segal group completion theorem (see \cite{ThomasNote} for a modern proof of the latter).
\end{remark}

 We can define a generalized ``free $\Z$-module'' functor in any distributive $\infty$-category as follows.
 Let $\FFree_\N$ (\resp{} $\FFree_\Z$) be the full subcategory of $\CMon(\Set)$ spanned by $(\N^n,+)$ (\resp{} by $(\Z^n,+)$) for $n\geq 0$.
 If $\scr C$ is an $\infty$-category with finite products, we shall denote by
 \[
 \Mod_\N(\scr C)\subset \Fun(\FFree_\N^\op,\scr C)\quad\text{and}\quad\Mod_\Z(\scr C)\subset\Fun(\FFree_\Z^\op,\scr C)
 \]
 the full subcategories of finite-product-preserving functors. The objects of $\Mod_\N(\scr C)$ are called \emph{strictly commutative monoids} in $\scr C$. 
 Since $\FFree_\N$ is semiadditive \cite[Definition 6.1.6.13]{HA}, the $\infty$-category $\Mod_\N(\scr C)$ is also semiadditive \cite[Corollary 2.4]{GGN} and the forgetful functor
 \[
 \Mod_\N(\scr C)\to \scr C, \quad M\mapsto M(\N),
 \]
  factors uniquely through the $\infty$-category $\CMon(\scr C)$ \cite[Corollary 2.5(iii)]{GGN}. 
 Similarly, as $\FFree_\Z$ is additive, the $\infty$-category $\Mod_\Z(\scr C)$ is additive and the forgetful functor $\Mod_\Z(\scr C)\to\scr C$ factors uniquely through $\CMon^\gp(\scr C)$.

 Assume now that $\scr C$ is distributive. The forgetful functors $\Mod_\N(\scr C)\to \scr C$ and $\Mod_\Z(\scr C) \to \scr C$ then
 preserve limits and sifted colimits, hence admit left adjoints
\[
\N\colon \scr C\to \Mod_\N(\scr C)\quad\text{and}\quad \Z\colon \scr C\to \Mod_\Z(\scr C).
\]
Since the $\infty$-categories $\Mod_\N(\scr C)$ and $\Mod_\Z(\scr C)$ are pointed, we also have reduced versions $\tilde\N X$ and $\tilde\Z X$ when $X$ is a pointed object of $\scr C$. More precisely, $\tilde \N$ is the unique colimit-preserving extension of $\N$ to $\scr C_*$, and similarly for $\tilde\Z$.

\begin{lemma}\label{lem:strict-gp}
	If $\scr C$ is an $\infty$-category with finite products, the square
 \begin{tikzmath}
 	\diagram{\Mod_\Z(\scr C) & \Mod_\N(\scr C) \\
	\CMon^\gr(\scr C) & \CMon(\scr C) \\};
	\arrows (11-) edge (-12) (11) edge node[left]{\textnormal{forget}} (21) (21-) edge[c->] (-22) (12) edge node[right]{\textnormal{forget}} (22);
 \end{tikzmath}
 is cartesian. If $\scr C$ is distributive, the top functor $\Mod_\Z(\scr C) \to \Mod_\N(\scr C)$ admits a left adjoint such that the following square commutes:
   \begin{tikzmath}
   	\diagram{\Mod_\Z(\scr C) & \Mod_\N(\scr C) \\
  	\CMon^\gr(\scr C) & \CMon(\scr C)\rlap. \\};
  	\arrows (11-) edge[<-] (-12) (11) edge node[left]{\textnormal{forget}} (21) (21-) edge[<-] node[above]{$\gp$} (-22) (12) edge node[right]{\textnormal{forget}} (22);
   \end{tikzmath}
	In particular, $\tilde\Z X\simeq (\tilde\N X)^\gp$ for any $X\in\scr C_*$.
\end{lemma}

\begin{proof}
	Since $\FFree_\N$ is semiadditive and $\FFree_\Z$ is additive, the forgetful functors $\CMon^\gp(\scr C)\to \CMon(\scr C)\to\scr C$ induce equivalences
	\begin{align*}
	\Mod_\N(\scr C)&\simeq \Mod_\N(\CMon(\scr C)), \\
	\Mod_\Z(\scr C)&\simeq \Mod_\Z(\CMon^\gp(\scr C)).
	\end{align*}
	The key point is that the $\infty$-category $\FFree_\Z$ is obtained from $\FFree_\N$ be group-completing the mapping spaces, so that $\FFree_\N\to\FFree_\Z$ is the universal finite-product-preserving functor to an additive $\infty$-category. Therefore, the forgetful functor
	\[
	\Mod_\Z(\scr C) \to \Mod_\N(\scr C)
	\]
	can be identified with the functor
	\[
	 \Mod_\Z(\CMon^\gp(\scr C))\simeq 
	\Mod_\N(\CMon^\gp(\scr C)) \to
	\Mod_\N(\CMon(\scr C)).
	\]
	This description immediately implies the claims: the first claim follows since a finite-product-preserving functor $M\colon \FFree_\N^\op\to \CMon(\scr C)$ lands in $\CMon^\gr(\scr C)$ if and only if $M(\N)$ is grouplike, and the second claims follows since group completion preserves finite products.
\end{proof}

We can describe free strictly commutative monoids more concretely using strict symmetric powers (see \S\ref{sec:symmetric_powers}):

\begin{lemma}\label{lem:N}
	Let $\scr C$ be a distributive $\infty$-category. Then the composite functor
	\[
	\scr C \xrightarrow{\N} \Mod_\N(\scr C) \xrightarrow{\mathrm{forget}} \scr C
	\]
	is given by $\displaystyle X\mapsto \coprod_{d\geq 0}\Sym^d X$.
\end{lemma}

\begin{proof}
	It suffices to check this for the universal $X$, which lives in the distributive $\infty$-category $\Fun(\Fin,\scr S)$.
	We may thus assume $\scr C=\scr S$. In that case, the forgetful functor $\Mod_\N(\scr S)\to \scr S$ is modeled by the right Quillen functor $\Mod_\N(\Set_\Delta) \to \Set_\Delta$ \cite[Proposition 5.5.9.1]{HTT}, whose left adjoint is given by the desired formula. Since the functor $\Sym^d$ on $\scr S$ can be computed using symmetric powers of CW complexes, this completes the proof.
\end{proof}

\begin{remark}
	Lemma~\ref{lem:N} shows that the endofunctor $X\mapsto \coprod_{d\geq 0}\Sym^d X$ of any distributive $\infty$-category has a canonical structure of monad. Its multiplication involves a canonical equivalence $\Sym^d(X\amalg Y) \simeq \coprod_{e+f=d}\Sym^eX\times \Sym^fY$ and a canonical map $\Sym^d\Sym^e X \to \Sym^{de} X$.
\end{remark}

If $\scr C$ is presentable and $\scr A$ is a small $\infty$-category with finite products, we have 
\[
\Fun^\times(\scr A,\scr C) \simeq \scr C\otimes \Fun^\times(\scr A,\scr S)
\] 
where $\otimes$ denotes the tensor product of presentable $\infty$-categories (this follows immediately from \cite[Proposition 4.8.1.17]{HA}). 
Hence, the presentable $\infty$-category $\Mod_\Z(\scr C)$ of grouplike strictly commutative monoids in $\scr C$ is a module over $\Mod_\Z(\scr S)$, which is the $\infty$-category of connective $H\Z$-modules. 
For $X\in \Mod_\Z(\scr C)$ and $A$ a connective $H\Z$-module, we will write $X\otimes A$ for the result of the action of $A$ on $X$.
Note that the construction $X\mapsto X\otimes A$ is preserved by any colimit-preserving functor $f\colon\scr C\to\scr D$. 

In particular, if $\scr C$ is distributive and $X\in\scr C$, one can form the strictly commutative monoid $\Z X\otimes A$ for any connective $H\Z$-module $A$, which can be described more concretely as follows.
Any connective $H\Z$-module $A$ can be obtained from $\Z$ in the following steps:
\begin{enumerate}
	\item take finite products of copies of $\Z$ to get finitely generated free $\Z$-modules;
	\item take filtered colimits of finitely generated free $\Z$-modules to get arbitrary flat $\Z$-modules;
	\item take colimits of simplicial diagrams of free $\Z$-modules to get arbitrary connective $H\Z$-modules \cite[Lemma 5.5.8.13]{HTT}.
\end{enumerate}
Since the forgetful functor $\Mod_\Z(\scr C)\to\scr C$ preserves finite products and sifted colimits, the object $\Z X\tens A$ in $\scr C$ can be obtained from $\Z X$ using the same steps.

In the distributive $\infty$-category $\scr S$, $\Z X\tens A$ has its ``usual'' meaning. For instance, if $A$ is an abelian group, then $\tilde\Z S^p\tens A$ is an Eilenberg--Mac Lane space $K(A,p)$.

\section{Sheaves with transfers}
\label{sec:transfers}

Let $S$ be a Noetherian scheme, $\scr C\subset\Sch_S$ an admissible category consisting of separated $S$-schemes, and $R$ a commutative ring.
We denote by $\Cor(\scr C,R)$ the additive category whose objects are those of $\scr C$ and whose morphisms are the finite correspondences with coefficients in $R$ \cite[\S 9]{CD}.\footnote{In \emph{loc.\ cit.}, it is assumed that $R\subset \Q$. For general $R$, we define $\Cor(\scr C,R)$ by extending scalars from the largest possible subring of $\Q$.}
 We denote by $\PSh^\ast(\scr C)$ the $\infty$-category of pointed presheaves on $\scr C$, by 
 \[
 \PSh^\tr(\scr C,R) = \Fun^\times(\Cor(\scr C,R)^\op, \scr S)
 \]
 the $\infty$-category of presheaves with $R$-transfers, and by
\[R_\tr: \PSh^\ast(\scr C)\rightleftarrows \PSh^\tr(\scr C,R):u_\tr\]
the free–forgetful adjunction.
The functor $u_\tr$ preserves limits and sifted colimits and factors through the $\infty$-category $\CMon^\gr(\PSh(\scr C))$; in fact, it factors through the $\infty$-category of grouplike strictly commutative monoids, using the finite-product-preserving functor 
\[
\FFree_\Z \to \Cor(\scr C,R), \quad \Z^n \mapsto S^{\amalg n}.
\]
Since finite products and finite coproducts coincide in semiadditive $\infty$-categories, the functor
\[u_\tr\colon\PSh^\tr(\scr C,R)\to\Mod_\Z(\PSh(\scr C))\]
preserves all colimits.

For $\tau$ a topology on $\scr C$, we denote by $\Shv_\tau^\tr(\scr C,R)$ the $\infty$-category of $\tau$-sheaves with $R$-transfers on $\scr C$, and by $\Shv_\tau^\tr(\scr C,R)_{\A^1}$ the $\infty$-category of homotopy invariant $\tau$-sheaves with $R$-transfers on $\scr C$; these are defined by the cartesian squares
\begin{tikzmath}
	\def\colsep{1.5em}
	\diagram{\Shv_\tau^\tr(\scr C,R)_{\A^1} & \Shv_\tau^\tr(\scr C,R) & \PSh^\tr(\scr C,R) \\ \Shv_\tau^\ast(\scr C)_{\A^1} & \Shv_\tau^\ast(\scr C) & \PSh^\ast(\scr C)\rlap. \\};
	\arrows (11-) edge[c->] (-12) (12-) edge[c->] (-13) (21-) edge[c->] (-22) (22-) edge[c->] (-23) (11) edge node[left]{$u_\tr$} (21) (12) edge node[left]{$u_\tr$} (22) (13) edge node[left]{$u_\tr$} (23);
\end{tikzmath}
By \cite[Proposition 5.5.4.15]{HTT}, the $\infty$-categories  $\Shv_\tau^\tr(\scr C,R)$ and  $\Shv_\tau^\tr(\scr C,R)_{\A^1}$ are presentable and there exist localization functors
\begin{gather*}
a_\tau^\tr\colon \PSh^\tr(\scr C,R)\to \Shv_\tau^\tr(\scr C,R),\\
L_{\tau,\A^1}^\tr\colon \PSh^\tr(\scr C,R) \to \Shv_\tau^\tr(\scr C,R)_{\A^1}.
\end{gather*}
Furthermore, by \cite[Proposition 5.4.6.6]{HTT}, the forgetful functors $u_\tr$ in the above diagram are accessible. Since they preserve limits, they admit left adjoint functors, which we will denote by $R_\tr$ (it will always be clear from the context which category $R_\tr$ is defined on).

We say that a topology $\tau$ on $\scr C$ is \emph{compatible with $R$-transfers} if for any presheaf with $R$-transfers $\scr F$ on $\scr C$, the canonical map
\[a_\tau u_\tr\scr F\to u_\tr a_\tau^\tr\scr F\]
is an equivalence. The following lemma shows that $\tau$ is compatible with transfers if and only if it is weakly compatible with transfers in the sense of~\cite[Definition 9.3.2]{CD}. For example, it follows from~\cite[Proposition 9.3.3]{CD} that the Nisnevich and étale topologies are compatible with any transfers on any admissible category.

\begin{lemma}\label{lem:compatible}
	A topology $\tau$ on $\scr C$ is compatible with $R$-transfers if and only if, for every $\tau$-covering sieve $U\into X$, the morphism
	\[a_\tau u_\tr R_\tr (U_+) \to a_\tau u_\tr R_\tr(X_+)\]
	is an equivalence in $\Shv_\tau^\ast(\scr C)$.
\end{lemma}

\begin{proof}
	If $\tau$ is compatible with transfers, then for any $\scr F\in\PSh^\ast(\scr C)$, 
	\[a_\tau u_\tr R_\tr \scr F\simeq u_\tr a_\tau^\tr R_\tr\scr F\simeq u_\tr R_\tr a_\tau\scr F.\]
	Since $a_\tau(U_+)\simeq a_\tau(X_+)$, this proves the ``only if'' direction.
	
	Conversely, define
	\[E=\{R_\tr U_+\to R_\tr X_+\suchthat U\into X\text{ is a $\tau$-covering sieve}\}\]
	so that $\Shv_\tau^\tr(\scr C,R)\subset\PSh^\tr(\scr C,R)$ is the subcategory of $E$-local objects, and suppose that the functor $a_\tau u_\tr$ sends elements of $E$ to equivalences.
	Let $\bar E$ be the strong saturation of $E$, \ie, the smallest class of morphisms containing $E$, satisfying the 2-out-of-3 property, and closed under colimits in $\Fun(\Delta^1,\PSh^\tr(\scr C,R))$.
	By~\cite[Proposition 5.5.4.15(4) and Proposition 5.2.7.12]{HTT}, the localization functor $a_\tau^\tr\colon\PSh^\tr(\scr C,R)\to\Shv_\tau^\tr(\scr C,R)$ is the universal functor sending elements of $\bar E$ to equivalences. 
	 We claim that $a_\tau u_\tr$ sends morphisms in $\bar E$ to equivalences. It will suffice to show that the class of morphisms $f$ such that $a_\tau u_\tr (f)$ is an equivalence is closed under the 2-out-of-3 property (which is obvious) and colimits. The functor $u_\tr\colon\PSh^\tr(\scr C,R)\to \PSh^\ast(\scr C,R)$ does not preserve colimits, but it preserves sifted colimits and transforms finite coproducts into finite products. Since $a_\tau$ is left exact and any colimit can be built out of finite coproducts and sifted colimits, this proves the claim. Thus, there exists a functor $f\colon\Shv_\tau^\tr(\scr C,R)\to\Shv_\tau^\ast(\scr C)$ making the diagram
	\begin{tikzmath}
		\diagram{\Shv_\tau^\tr(\scr C,R) & \PSh^\tr(\scr C,R) & \Shv_\tau^\tr(\scr C, R) \\ 
		\Shv_\tau^\ast(\scr C) & \PSh^\ast(\scr C) & \Shv_\tau^\ast(\scr C) \\};
		\arrows (11-) edge[c->] (-12) (12-) edge node[above]{$a_\tau^\tr$} (-13)
		(11) edge node[left]{$u_\tr$} (21) (12) edge node[left]{$u_\tr$} (22) (13) edge node[right]{$f$} (23)
		(21-) edge[c->] (-22) (22-) edge node[below]{$a_\tau$} (-23);
	\end{tikzmath}
	commute. Since the horizontal compositions are the identity, $f\simeq u_\tr$ and $a_\tau u_\tr\simeq u_\tr a_\tau^\tr$.
\end{proof}

\begin{lemma}\label{lem:sifted0}
	Suppose that $\tau$ is compatible with $R$-transfers. Then the square
	\begin{tikzcentered}
		\diagram{ \PSh^\tr(\scr C,R) & \Shv^\tr_\tau(\scr C,R)_{\A^1} \\ \Mod_\Z(\PSh(\scr C)) & \Mod_\Z(\Shv_\tau(\scr C)_{\A^1}) \\};
		\arrows (11-) edge node[above]{$L_{\tau,\A^1}^\tr$} (-12) (11) edge node[left]{$u_\tr$} (21) (21-) edge node[above]{$L_{\tau,\A^1}$} (-22) (12) edge node[right]{$u_\tr$} (22);
	\end{tikzcentered}
	commutes. 
\end{lemma}

\begin{proof}
	Consider the diagram
	\begin{tikzcentered}
		\diagram{\Shv^\tr_\tau(\scr C,R)_{\A^1} & \PSh^\tr(\scr C,R) & \Shv^\tr_\tau(\scr C,R)_{\A^1} \\ \Mod_\Z(\Shv_\tau(\scr C)_{\A^1}) & \Mod_\Z(\PSh(\scr C)) & \Mod_\Z(\Shv_\tau(\scr C)_{\A^1})\rlap. \\};
		\arrows (11-) edge[c->] (-12) (21-) edge[c->] (-22) (11) edge node[left]{$u_\tr$} (21)(12-) edge node[above]{$L_{\tau,\A^1}^\tr$} (-13) (12) edge node[right]{$u_\tr$} (22) (22-) edge node[above]{$L_{\tau,\A^1}$} (-23) (13) edge[dashed] node[right]{$f$} (23);
	\end{tikzcentered}
	 It will suffice to show that a functor $f$ exists as indicated. Define
	\begin{align*}
		E_\tau &= \{R_\tr U_+\to R_\tr X_+\suchthat U\into X\text{ is a $\tau$-covering sieve in $\scr C$}\},\\
		E_{\A^1} &= \{R_\tr(X\times\A^1)_+\to R_\tr X_+\suchthat X\in\scr C\},
	\end{align*}
	so that $\Shv_\tau^\tr(\scr C,R)_{\A^1}\subset \PSh^\tr(\scr C,R)$ is the full subcategory of $(E_\tau\cup E_{\A^1})$-local objects.
	The functor $L_{\tau,\A^1}u_\tr$ carries morphisms in $E_\tau$ and $E_{\A^1}$ to equivalences: for $E_\tau$, this is because $\tau$ is compatible with $R$-transfers and for $E_{\A^1}$ it is because $u_\tr R_\tr(X\times\A^1)_+\to u_\tr R_\tr X_+$ is an $\A^1$-homotopy equivalence (see the last part of the proof of~\cite[Theorem 1.7]{MEMS}).
	By~\cite[Proposition 5.5.4.20]{HTT}, there exists a functor $f$ making the above diagram commutes. 
\end{proof}

\begin{corollary}\label{cor:sifted}
	Suppose that $\tau$ is compatible with $R$-transfers. Then the forgetful functor $u_\tr\colon\Shv^\tr_\tau(\scr C,R)_{\A^1}\to \Mod_\Z(\Shv_\tau(\scr C)_{\A^1})$ preserves colimits.
\end{corollary}

\begin{proof}
	 This follows immediately from Lemma~\ref{lem:sifted0}. 
\end{proof}

The $\infty$-category $\Shv_{\tau}^\tr(\scr C,R)_{\A^1}$ is tensored over connective $HR$-modules. For $p\geq q\geq 0$ and $A$ a connective $HR$-module, the \emph{motivic Eilenberg--Mac Lane space} 
\[
K(A(q),p)_{\scr C}\in\Shv_\Nis^\ast(\scr C)_{\A^1}
\]
is defined by \[K(A(q),p)_{\scr C}=u_\tr (R_\tr S^{p,q}\tens A),\]
where $S^{p,q}\in\Shv_\Nis^\ast(\scr C)_{\A^1}$ is the motivic $p$-sphere of weight $q$.
Although this construction depends on the coefficient ring $R$ in general, it does not if either the schemes in $\scr C$ are regular or if the positive residual characteristics of $S$ are invertible in $R$ \cite[Remark 9.1.3(3)]{CD}; the latter will always be the case in what follows. 

\section{The étale homotopy type of motivic Eilenberg--Mac Lane spaces} 
\label{sec:EML}

Let $k$ be a separably closed field, $l\neq\car k$ a prime number, and $\scr C\subset \Sch_k$ an admissible category.
One defect of the $\Z/l$-local étale homotopy type functor $\Et_l\colon \Shv_\Nis(\scr C)_{\A^1}\to \Pro(\scr S)_{\Z/l}$ is that it does not preserve finite products and hence does not preserve commutative monoids. We have seen in Lemma~\ref{lem:Etdistributive} that $\Et_l$ preserves finite products between compact motivic spaces, but motivic Eilenberg–Mac Lane spaces are certainly not compact. We will fix this problem by constructing a factorization
	\[\Shv_\Nis(\scr C)_{\A^1}\xrightarrow{\Etd} \scr E\to\Pro(\scr S)_{\Z/l} \]
	of $\Et_l$ such that $\Etd$ is distributive and $\scr E$ is a ``close approximation'' of $\Pro(\scr S)_{\Z/l}$ by a distributive $\infty$-category (see Definition~\ref{def:distributive}). For our applications, we will also need $\Etd$ to factor through $\Shv_\h^\wedge(\Sch_k^\ft)_{\A^1}$. 
	The fact that the latter $\infty$-category may not be compactly generated explains some of the complexity of the following construction.

\begin{construction}\label{construction}
 Let $i\colon\scr C\into\Sch_k^\ft$ be the inclusion and let $i_!\colon\Shv_\Nis(\scr C)\to \Shv_\Nis(\Sch_k^\ft)$ be its colimit-preserving extension. Since $\scr C$ is admissible and in particular closed under finite products, the functor $i_!$ is distributive. Note that we have a commuting triangle
	\begin{tikzmath}
		\def\colsep{-.5em}
		\diagram{\Shv_\Nis(\scr C) & & \Shv_\h(\Sch_k^\ft) \\ & \Pro(\scr S)\rlap, & \\};
		\arrows (11-) edge node[above]{$a_\h i_!$} (-13) (11) edge node[left=4pt]{$\LPi_\infty^\h$} (22)
		(13) edge node[right=4pt]{$\LPi_\infty^\h$} (22);
	\end{tikzmath}
	where $\LPi_\infty^\h$ was defined in \S\ref{sec:homotopy_type_of_symmetric_powers}, and that the functor $\LPi_\infty^\h$ on the right is simply $\Pi_\infty$, in the sense of Remark~\ref{rmk:pi_infty}.
	For any $\infty$-topos $\scr X$, we have a commutative square
	\begin{tikzmath}
		\diagram{\scr X & \scr X_{\leq n} \\ \Pro(\scr S) & \Pro(\scr S_{\leq n})\rlap, \\};
		\arrows(11-) edge node[above]{$\tau_{\leq n}$} (-12) (21-) edge node[above]{$\tau_{\leq n}$} (-22) (11) edge node[left]{$\Pi_\infty$} (21) (12) edge node[right]{$\Pi_n$} (22);
	\end{tikzmath}
	where the horizontal maps are given by truncation and $\Pi_n=\tau_{\leq n}\circ \Pi_\infty$. We can therefore factor the $\Z/l$-local shape functor $L_{\Z/l}\Pi_\infty\colon\scr X\to\Pro(\scr S)_{\Z/l}$ as
	\[\scr X\rightarrow \scr X^\wedge\xrightarrow{\tau_{\leq *}}\lim_n\scr X_{\leq n}\to\lim_n\Pro(\scr S_{\leq n})_{\Z/l}\simeq \Pro(\scr S)_{\Z/l},\]
	where the last equivalence is~\eqref{eqn:postnikov}.
	Since truncations preserve colimits and finite products, the functor $\tau_{\leq *}$ is distributive. Applying this procedure to the $\infty$-topos $\scr X=\Shv_\h(\Sch_k^\ft)$, we obtain a factorization of $L_{\Z/l}\LPi_\infty^\h$ as
	\[\Shv_\Nis(\scr C)\xrightarrow{a_\h^\wedge i_!}\Shv_\h^\wedge(\Sch_k^\ft)\xrightarrow{\tau_{\leq *}}\lim_n\Shv_\h(\Sch_k^\ft)_{\leq n}\to\lim_n\Pro(\scr S_{\leq n})_{\Z/l}\simeq \Pro(\scr S)_{\Z/l}.\]
	
	Let $\Pro'(\scr S)$ be the smallest full subcategory of $\Pro(\scr S)$ containing $\Pi_\infty^\h X$ for every $k$-scheme of finite type $X$ and closed under finite products, finite colimits, and retracts. 
	We similarly define the $\infty$-categories $\Pro'(\scr S)_{\Z/l}$, $\Pro'(\scr S_{\leq n})$, and $\Pro'(\scr S_{\leq n})_{\Z/l}$ to be generated by the respective localizations of $\Pi_\infty^\h X$.
	Recall that the localization functors $L_{\Z/l}$ and $\tau_{\leq n}$ preserve finite products (Proposition~\ref{prop:finiteprod}).
	By Remarks~\ref{rmk:finitevs} and \ref{rmk:finitehomology}, the pro-spaces in $\Pro'(\scr S)$ satisfy the assumption of Proposition~\ref{prop:Kunneth2}, so that the localization functor $L_{\Z/l,\leq n}\colon \Pro'(\scr S)\to \Pro'(\scr S_{\leq n})_{\Z/l}$ also preserves finite products.
	It follows that finite products distribute over finite colimits in each of these $\infty$-categories, so that their ind-completions are distributive.
	
	Consider the colimit-preserving functor
	\[L_{\Z/l,\leq n}\LPi_\infty^\h\colon\Shv_{\h}(\Sch_k^\ft)_{\leq n}\to\Pro(\scr S_{\leq n})_{\Z/l}.\]
	By Lemma~\ref{lem:compact}, 
	the $\infty$-category $\Shv_\h(\Sch_k^\ft)_{\leq n}$ is compactly generated by representables, so that $(\Shv_{\h}(\Sch_k^\ft)_{\leq n})^\omega$ is generated by representables under finite colimits and retracts. Hence, this functor restricts to a functor
	\[L_{\Z/l,\leq n}\LPi_\infty^\h\colon(\Shv_{\h}(\Sch_k^\ft)_{\leq n})^\omega\to\Pro'(\scr S_{\leq n})_{\Z/l}\]
	that preserves finite colimits. It also preserves finite products by Proposition~\ref{prop:etalekunneth}, so that the induced functor on ind-completions
	\[\Shv_\h(\Sch_k^\ft)_{\leq n}\to\Ind(\Pro'(\scr S_{\leq n})_{\Z/l})\]
	is distributive.
	Thus, the composition
	\[\Shv_\Nis(\scr C)\xrightarrow{\tau_{\leq *}a_\h^\wedge i_!}\lim_n\Shv_\h(\Sch_k^\ft)_{\leq n}\to\lim_n\Ind(\Pro'(\scr S_{\leq n})_{\Z/l})\]
	is distributive, and it inverts $X\times\A^1\to X$ for every $X\in \scr C$ since $L_{\Z/l}\Pi_\infty^\h$ does, so it induces a distributive functor \[\Etd\colon \Shv_\Nis(\scr C)_{\A^1}\to \lim_n\Ind(\Pro'(\scr S_{\leq n})_{\Z/l}).\] By construction, $\Et_l$ is the composition of $\Etd$ and the colimit functor
	\[
	\pushQED{\qed}
	\lim_n\Ind(\Pro'(\scr S_{\leq n})_{\Z/l})\to\lim_n\Pro(\scr S_{\leq n})_{\Z/l}\simeq\Pro(\scr S)_{\Z/l}.\qedhere
	\popQED
	\]
\end{construction}

The next theorem is our étale version of~\cite[Proposition 3.41]{MEMS}. We point out that the category $\scr C$ in the statement below need \emph{not} be closed under symmetric powers, so the theorem applies directly with $\scr C$ the category of smooth separated $k$-schemes with no need for resolutions of singularities. 

\begin{theorem}\label{thm:main2}
	Let $k$ be an algebraically closed field of characteristic exponent $e$, $l\neq e$ a prime number, and $\scr C\subset\Sch_k$ an admissible subcategory consisting of semi-normal separated schemes. Then for any pointed object $X$ in $\Shv_{\Nis}(\scr C)_{\A^1}$ and any connective $H\Z[1/e]$-module $A$, there is an equivalence
	\[\theta_{X,A}\colon \tilde\Z\Etd X\tens A \simeq \Etd u_\tr(\Z_\tr X\tens A)\]
	of grouplike strictly commutative monoids in $\lim_n\Ind(\Pro'(\scr S_{\leq n})_{\Z/l})$, natural in $X$ and $A$, with the following properties:
	\begin{enumerate}
		\item the triangle
		\begin{tikzmath}
			\def\colsep{4em}
			\diagram{\Etd X & \tilde\Z\Etd X\tens \Z[1/e] \\ & \Etd u_\tr \Z[1/e]_\tr X \\};
			\arrows (11-) edge node[above]{$1$} (-12) (11) edge node[below left]{$\Etd(\textnormal{unit})$} (22) (12) edge node[right]{$\theta_{X,\Z[1/e]}$} node[left]{$\simeq$} (22);
		\end{tikzmath}
		is commutative;
		\item given $X$, $Y$, $A$, and $B$, the square
		\begin{tikzmath}
			\def\colsep{-5em}
			\diagram{(\tilde\Z\Etd X\tens A)\wedge (\tilde\Z\Etd Y\tens B) & \\ & \tilde\Z\Etd (X\wedge Y)\tens (A\tens B) \\ 
		\Etd u_\tr(\Z_\tr X\tens A)\wedge \Etd u_\tr(\Z_\tr Y\tens B)  & \\  & \Etd u_\tr(\Z_\tr (X\wedge Y)\tens (A\tens B))\\};
			\arrows (11) edge node[left]{$\theta_{X,A}\wedge \theta_{Y,B}$} node[right]{$\simeq$} (31) (22) edge node[right]{$\theta_{X\wedge Y,A\tens B}$}  node[left]{$\simeq$} (42)
			(11) edge (22) (31) edge (42);
		\end{tikzmath}
		is commutative.
	\end{enumerate}
\end{theorem}

\begin{proof}
	Any $k$-scheme of finite type is Zariski-locally quasi-projective, so we can assume that the schemes in $\scr C$ are quasi-projective without changing the categories and functors involved.
	As $X$ varies, the source and target of $\theta_{X,A}$ are functors taking values in strictly commutative monoids in $\lim_n\Ind(\Pro'(\scr S_{\leq n})_{\Z/l})$, and as such they preserve colimits: for the left-hand side this is clear and for the right-hand side it follows from Corollary~\ref{cor:sifted}.
	In particular, these functors are left Kan extended from their restriction to $\scr C$.
	 To show the existence of $\theta_{X,A}$, it will therefore suffice to define $\theta_{X,A}$ for $X$ representable, \ie, $X=L_{\Nis,\A^1}r(Z)_+$ for some $Z\in\scr C$, and this construction should be natural in $Z$ and $A$.
	Furthermore, since $u_\tr$ is $H\Z$-linear and $\Etd$ is distributive, we have
	\[
	\Etd u_\tr(\Z_\tr X\tens A) \simeq \Etd u_\tr(\Z_\tr X)\otimes A
	\]
	so we can assume $A=\Z[1/e]$.
	Let $i\colon\scr C\into \Sch_k^\ft$ be the inclusion.
	By the motivic Dold–Thom theorem \cite[Theorem 3.7]{MEMS} (see also \cite[Theorem 6.8]{Suslin:1996}), there is an equivalence
	\[u_\tr\Z[1/e]_\tr Z_+\simeq a_\Nis((\coprod_{d\geq 0}i^*r(\Sym^dZ))^\gp[1/e])\]
	of pointed presheaves on $\scr C$, natural in $Z$. Note that the validity of this formula does not depend on the scheme $\Sym^dZ$ belonging to $\scr C$. By Lemma~\ref{lem:sifted0} and the distributivity of $L_{\Nis,\A^1}$, we obtain equivalences
	\[u_\tr\Z[1/e]_\tr X\simeq L_{\Nis,\A^1}((\coprod_{d\geq 0}i^*r(\Sym^dZ))^\gp[1/e])\simeq (\coprod_{d\geq 0}L_{\Nis,\A^1} i^*r(\Sym^dZ))^\gp[1/e]\]
	in $\Shv_\Nis(\scr C)_{\A^1}$. Since $\Etd $ preserves group completions of commutative monoids (Lemma~\ref{lem:grp}) and colimits,
	\[\Etd u_\tr\Z[1/e]_\tr X\simeq (\coprod_{d\geq 0}\Etd L_{\Nis,\A^1} i^* r(\Sym^dZ))^\gp[1/e].\]
	On the other hand, by Lemmas~\ref{lem:strict-gp} and~\ref{lem:N} and since $\Etd L_{\Nis,\A^1}$ commutes with $\Sym^d$,
	\[\tilde\Z\Etd X\tens\Z[1/e]\simeq (\coprod_{d\geq 0}\Etd L_{\Nis,\A^1}\Sym^dr_\Nis(Z))^\gp[1/e].\]
 We then define $\theta_{X,\Z[1/e]}\colon \tilde\Z\Etd X\tens\Z[1/e]\to \Etd u_\tr\Z[1/e]_\tr X$ to be the map induced by the obvious canonical map
	\[\phi\colon \Sym^d r_\Nis(Z)\to i^* r(\Sym^dZ)\]
	in $\Shv_\Nis(\scr C)_{\leq 0}$.
	To show that $\theta_{X,\Z[1/e]}$ is an equivalence, it suffices to show that $\Etd L_{\Nis,\A^1}(\phi)$ is an equivalence. By definition of $\Etd$, the functor $\Etd L_{\Nis,\A^1}$ factors through $a_{\h}^\wedge i_!\colon\Shv_\Nis(\scr C)\to\Shv_\h^\wedge(\Sch_k^\ft)$, so it suffices to show that $a_{\h}^\wedge i_!(\phi)$ is an equivalence. This follows from Corollary~\ref{cor:quotient2} and Proposition~\ref{prop:dejong} (since $k$ is perfect and every smooth $k$-scheme is Zariski-locally in $\scr C$).

	The strategy to prove (1) and (2) is the following: we first reduce as above to the representable case, where the statements follow from properties of the motivic Dold--Thom equivalence. 
	For (1), we may assume that $X$ is represented by $Z\in\scr C$. Then the adjunction map $X\to u_\tr\Z[1/e]_\tr X$ corresponds, through the Dold--Thom equivalence, to the map \[Z_+\simeq\Sym^0Z\amalg \Sym^1Z\into \coprod_{d\geq 0}\Sym^dZ\to (\coprod_{d\geq 0}\Sym^dZ)^\gp[1/e],\]
	which proves the result.
	For~(2), 
	we may assume that $X$ and $Y$ are represented by $Z$ and $W$ and that $A=B=\Z[1/e]$.
	Moreover, we can replace $\tilde \Z$ with $\tilde\N$.
	It then suffices to note that the pairing
	\[u_\tr\Z[1/e]_\tr X\wedge u_\tr \Z[1/e]_\tr Y\to u_\tr \Z[1/e]_\tr(X\wedge Y)\]
	arising from the monoidal structure of $\Z[1/e]_\tr$ is induced, via the Dold–Thom equivalence, by the obvious maps $\Sym^a Z\times\Sym^b W\to\Sym^{ab}(Z\times W)$.
\end{proof}

For $A$ a connective $H\Z$-module, let $A_l$ be the pro-$H\Z$-module $\lim_n A/l^n$.
Note that $A_l\simeq A$ if $A$ admits an $H\Z/l^n$-module structure for some $n\geq 1$.

\begin{lemma}\label{lem:EMLcompletion}
	Let $A$ be a connective $H\Z$-module and $p\geq 1$. Then
	\[L_{\Z/l}K(A,p) \simeq \tau_{<\infty}K(A_l,p).\]
\end{lemma}

\begin{proof}
	We can assume $A$ truncated, as both sides preserve the limit of the Postnikov tower of $A$.
	Then $A$ is a finite product of $H\Z$-modules of the form $B[i]$ with $B$ discrete and $i\geq 0$.
	Since $K(B[i],p)\simeq K(B,p+i)$ and since both sides preserve finite products (see Proposition~\ref{prop:finiteprod}), we can assume $A$ discrete.
	By the principal fibration lemma \cite[III, 3.6]{BK}, we can reduce to the case $p=1$ and $A$ free abelian. 
	In this case we have $\pi_nL_{\Z/l}K(A,1)=0$ for $n\geq 2$ \cite[IV, Lemma 4.4]{BK} and $\pi_1L_{\Z/l}K(A,1)\simeq A_l$ by the proof of \cite[IV, Lemma 2.4]{BK}.
\end{proof}

Given $q\in\Z$, we let $A_l(q)=A_l\tens_{\Z_l}T_l\mu^{\tens q}$ (which is noncanonically isomorphic to $A_l$). For example, $\Z/l^n(q)\simeq \mu_{l^n}^{\tens q}$.

\begin{theorem}\label{thm:EML}
Let $k$ be an algebraically closed field of characteristic exponent $e$, $l\neq e$ a prime number, and $\scr C\subset\Sch_k$ an admissible subcategory consisting of semi-normal separated schemes.
For any connective $H\Z[1/e]$-module $A$ and any integers $p,q$ with $p\geq 1$ and $p\geq q\geq 0$, there is a canonical equivalence
	\[\Et_l K(A(q),p)_{\scr C}\simeq \tau_{<\infty}K(A_l(q),p)\]
	of pointed objects in $\Pro(\scr S)_{\Z/l}$, natural in $A$, and $\Et_l$ preserves smash products between such spaces. 
	Furthermore, under these equivalences, 
	\begin{enumerate}
		\item $\Et_l$ sends the canonical map
		\[S^{p,q}\to K(\Z[1/e](q),p)_{\scr C}\]
		to the canonical map
		\[S^{p,q}_l\to K(\Z_l(q),p);\]
		\item $\Et_l$ sends the canonical map
		\[K(A(q),p)_{\scr C}\wedge K(B(s),r)_{\scr C}\to K((A\tens B)(q+s),p+r)_{\scr C}\]
		to the canonical map
		\[\tau_{<\infty}K(A_l(q),p)\wedge \tau_{<\infty}K(B_l(s),r)\to \tau_{<\infty}K((A\tens B)_l(q+s),p+r).\]
	\end{enumerate}
\end{theorem}

\begin{proof}
	By Theorem~\ref{thm:main2}, we have
	\[
	\Etd  K(A(q),p)_{\scr C}\simeq \tilde\Z\Etd S^{p,q}\tens A.
	\]
	By Proposition~\ref{prop:spheres} and the definition of $\Etd$ (see Construction~\ref{construction}), we have
	\[
	\Etd S^{p,q} 
	\simeq L_{\Z/l,\leq *} S^{p,q}_l,
	\]
	where $S^{p,q}_l\in\Pro'(\scr S)_{\Z/l}$ is considered as a constant ind-$\Z/l$-local pro-space. Thus,
	\begin{equation*}\label{eqn:indEML}
		\Etd  K(A(q),p)_{\scr C}\simeq \tilde\Z L_{\Z/l,\leq *} S^{p,q}_l\tens A.
	\end{equation*}
	Now we apply the colimit functor
	\[c\colon \lim_n\Ind(\Pro'(\scr S_{\leq n})_{\Z/l})\to\lim_n\Pro(\scr S_{\leq n})_{\Z/l}\simeq \Pro(\scr S)_{\Z/l}\]
	 to both sides. The left-hand side becomes $\Et_l K(A(q),p)_{\scr C}$, by definition of $\Etd$. 
	 We are thus reduced to proving that
	 \begin{equation}\label{eqn:final}
	 c(\tilde \Z L_{\Z/l,\leq *}S^{p,q}_l\otimes A) \simeq \tau_{<\infty} K(A_l(q),p).
	 \end{equation}
	 For this, we consider the following commutative diagram:
	\begin{tikzcentered}
		\diagram{\scr S^\omega & \Ind(\scr S^\omega) & \lim_n\Ind(\Pro'(\scr S_{\leq n})) & \lim_n\Ind(\Pro'(\scr S_{\leq n})_{\Z/l}) & \\ & \scr S & \Pro(\scr S_{<\infty}) & \Pro(\scr S)_{\Z/l}\rlap.\\};
		\arrows (11-) edge[c->] node[above]{$i$} (-12) (11) edge[right hook->] (22)
		(12-) edge node[above]{$j^\ind$} (-13) (12) edge node[right]{$c$} node[left]{$\simeq$} (22) (22-) edge node[above]{$j$} (-23) (13) edge node[right]{$c$} (23) (13-) edge node[above]{$L^\ind$} (-14) (14) edge node[right]{$c$} (24) (23-) edge node[above]{$L_{\Z/l}$} (-24);
	\end{tikzcentered}
	Note that $L_{\Z/l,\leq *}S^{p,0}_l$ is the image of $S^p$ by the top row of this diagram.
	 The functor $j^\ind$ is clearly distributive and $L^\ind$ is distributive by Proposition~\ref{prop:Kunneth2}. We therefore have equivalences
	\[
		c(\tilde \Z L_{\Z/l,\leq *}S^{p,0}_l\otimes A) \simeq c (\tilde\Z (L^\ind j^\ind iS^{p})\tens A)\simeq c L^\ind j^\ind (\tilde\Z iS^p\tens A)
		\simeq L_{\Z/l} jK(A,p),
	\]
	which concludes the case $q=0$ by Lemma~\ref{lem:EMLcompletion}.
	It also concludes the general case if one is willing to choose an isomorphism $\Z_l\simeq T_l\mu$.
	However, the twisting for $q>0$ makes the equivalence~\eqref{eqn:final} independent of such a choice, as a consequence of the following observation: if $\alpha$ is an automorphism of $\Z_l$, then the automorphism $\mathrm{id}^{\wedge (p-q)}\wedge \alpha^{\wedge q}$ of $S^{p,0}_l$ induces $\alpha^{\otimes q}$ on the top homology pro-group $\Z_l$.
	A similar argument shows that $\Et_l(X\wedge Y)\simeq \Et_l(X)\wedge\Et_l(Y)$ whenever $\Etd X$ and $\Etd Y$ belong to the essential image of $L^\ind j^\ind$. The remaining statements are easily deduced from properties (1) and (2) in Theorem~\ref{thm:main2}.
\end{proof}

In conclusion, let us emphasize the two most important special cases of Theorem~\ref{thm:EML}:

\begin{corollary}
	Let $k$ be an algebraically closed field of characteristic exponent $e$, $l\neq e$ a prime number, and $\scr C\subset\Sch_k$ an admissible subcategory consisting of semi-normal separated schemes.
	\begin{enumerate}
		\item For any $\Z[1/e]\subset\Lambda\subset \Z_{(l)}$, there is a canonical equivalence
		\[
		\Et_l K(\Lambda(q),p)_{\scr C} \simeq K(T_l\mu^{\otimes q}, p),
		\]
		where $T_l\mu^{\otimes q}$ is the $l$-adic Tate module of $\mu^{\otimes q}$ viewed as a pro-group.
		\item For any $n\geq 1$, there is a canonical equivalence
		\[
		\Et_l K(\Z/l^n(q),p)_{\scr C} \simeq K(\mu_{l^n}^{\otimes q}, p).
		\]
		In particular, $\Et_l K(\Z/l^n(q),p)_{\scr C}$ is a constant pro-space.
	\end{enumerate}
	In both cases, the $\Z/l$-local étale homotopy type is already $\Z/l$-profinite.
\end{corollary}

\providecommand{\bysame}{\leavevmode\hbox to3em{\hrulefill}\thinspace}

\end{document}